\documentclass[11pt]{article}

\usepackage{amsmath,amssymb, amsthm,amsrefs}
\usepackage{ulem}
\usepackage{cite}

\DeclareMathOperator{\sol}{div}					
\DeclareMathOperator{\curl}{curl}		
	
\newcommand{\norm}[1]{\left\|#1\right\|}   
\newcommand{\ex}[1]{\mathrm{e}^{#1}}		  
\newcommand{\derivuj}[2]{\frac{\partial{#1}}{\partial{{#2}}}}

\newcommand{\derit}[1]{\derivuj{#1}{t}}
\newcommand{\derits}[1]{\partial_{t}{\left(#1\right)}}
\newcommand{\deritt}[1]{\frac{\partial^2{#1}}{\partial t^2}}
\newcommand{\abs}[1]{\left\lvert#1\right\rvert}        
\newcommand{\abso}[1]{\biggl\lvert#1\biggr\rvert}      
\newcommand{\field}[1]{\mathbb{#1}}
\def\det{\partial_t}
\def\R{\field{R}}      
\def\N{\field{N}}

\def\mfi{\varphi}
\def\sfi{\psi}
\def\pfi{\phi}
\def\S{\field{S}} 
\def\B{\mathcal{B}}
\def\Te{\mathcal{T}}
\newcommand{\rntou}[1]{\R^{#1}}    

\def\carka{\sp{\prime}}
\def\rtri{\rntou{3}}
\def\es{S^1}
\def\Id{\field{I}} 
\def\epsil{\varepsilon}
\newcommand{\vektor}[1]{{\mathbf{#1}}}
\def\ro{\varrho}
\def\teta{\vartheta}
\def\en{\vektor{n}}
\def\bfi{\boldsymbol\mfi}
\def\Phib{\boldsymbol\Phi}
\def\kve{\vektor{q}}
\def\ef{\vektor{f}}
\def\u{\vektor{u}}
\newcommand{\de}[1]{\mathrm{d}{#1}}     
\def\dt{\, \de{t}}	
\def\dx{\, \de{x}}	
\newcommand{\inth}[1]{\int\limits_{\partial\Omega}{#1}\, \de{S}  }
\newcommand{\ins}[1]{\int\limits_{\es}{#1}\dt  }
\newcommand{\inte}[1]{\int\limits_{\Omega}{#1}\dx  }
\newcommand{\INT}[1]{\ins{\inte{#1}}}
\def\INTb{\int\limits_{\es}\int\limits_{\Omega} }
\newcommand{\INTpsi}[1]{\ins{\inte{\left(#1\right)}\:\psi\: }}
\newcommand{\INTpsib}[1]{\ins{\inte{#1}\:\psi\: }}
\newcommand{\refx}[1]{(\ref{#1})}
\def\qqquad{\quad\quad\quad}

\newcommand{\elka}[2]{L^{#1}\left( \es; L^{#2}(\Omega)\right)}
\newcommand{\elkas}[2]{L^{#1}\left(L^{#2}\right)}
\def\sobo{W^{1,2}_0(\Omega)}
\def\sob{W^{1,2}(\Omega)}

\def\sobdp{W^{2,p}(\Omega)}
\def\soblps{W^{1,p}(\es\times\Omega)}
\def\sobdps{W^{2,p}(\es\times\Omega)}
\newcommand{\lebs}[1]{L^{#1}(\es\times\Omega)}
\newcommand{\Cw}[1]{C_{\mathrm{weak}}(\es,L^{#1}(\Omega)) }
\def\cinf{C^{\infty}}
\def\ksi{\zeta}
\def\wi{\vektor{w}^i}
\def\wj{\vektor{w}^j}
\def\unt{\vektor{\tilde{u}}_N}
\def\Nt{N_t}
\def\Nx{N_{x}}
\def\tteta{\tilde{\teta}}
\def\te{\tilde{e}}
\def\wiN{\left\{\wi\right\}_{i=1}^N}
\def\LwiN{ Lin\wiN }
\def\EN{E_{N}}
\def\ENt{E_{N}(t)}
\def\Ed{E_{\delta}}
\def\Edt{E_{\delta}(t)}
\newcommand{\fra}[2]{{#1}/{#2}}

\def\weak{ \rightharpoonup}
\def\weaks{\weak^{*}}
\def\strong{ \rightarrow}
\def\embed{\hookrightarrow}

\renewcommand{\emph}[1]{\textit{#1}}

\theoremstyle{plain}
\newtheorem{thm}{Theorem}
\newtheorem{lem}{Lemma}
\newtheorem{prop}{Proposition}
\newtheorem*{note}{Remark concerning the models}
\newtheorem*{LSFPT}{Schauder fixed point theorem}

\newtheorem*{notka}{Remark}

\newcommand{\elnor}[2]{\norm{#1}_{#2}} 
\newcommand{\elnorm}[3]{\elnor{#1}{\elkas{#2}{#3} } }

\title{Time-periodic solutions to the full Navier--Stokes--Fourier system with radiation on the boundary}

\author{\v{S}imon Axmann, Milan Pokorn\'{y}\\
\small{Charles University in Prague}\\
\small{Faculty of Mathematics and Physics}\\
\small {Mathematical Institute of Charles University} \\
\small{Sokolovsk\'a 83, 186 75 Praha 8, Czech Republic} \\
\small{e-mails: {\tt axmann.simon@seznam.cz}, {\tt pokorny@karlin.mff.cuni.cz}}}

\begin{document}

\maketitle

\vspace*{1cm}

\textbf{Abstract:} The Navier-Stokes-Fourier system is a well established model for describing the motion of viscous compressible heat-conducting fluids. We study the existence of time-periodic weak solutions and improve the result from \cite{FeTP} in the following sense: we extend the class of pressure functions (i.e. consider lower exponent $\gamma$) and include also the effect of radiation on the boundary.

\vspace*{1cm}

\section{Introduction}

We consider the following system of partial differential equations for three unknowns: the fluid density $\ro$, the velocity field $\u$, and the absolute temperature $\teta$; the succeeding identities represent the balance of mass, the balance of linear momentum, the balance of entropy and the conservation of the total energy, respectively.

\begin{gather}
	\derit{\ro} + \sol(\ro \u) = 0, \label{RK} \\ 
	\derit{(\ro\u)} + \sol(\ro \u\otimes\u) + \nabla p(\ro,\teta)=\sol\S(\teta,\nabla\u)+\ro\ef  , \\
	\derit{\bigl(\ro s(\ro,\teta)\bigr)} + \sol\bigl(\ro s(\ro,\teta)\u\bigr) +\sol\left( \frac{\kve(\teta,\nabla\teta)}{\teta} \right)=\sigma ,\\
	\frac{\de { } }{\dt}	\inte{ \left(    \frac{1}{2}\ro\abs{\u}^2 + \ro e(\ro,\teta) \right) } = \inte {\ro \ef\cdot \u} - \inth{\kve\cdot\en} .\label{TE}
\end{gather}
In order to close the system, we have to specify the so-called constitutive relations. The viscous stress $\S$ is assumed to satisfy the Stokes law for Newtonian fluid
\begin{equation}
\S (\teta,\nabla\u) = \mu(\teta)\left( \nabla\u +\nabla^T\u- \frac{2}{3}\sol\u \Id \right) +\eta(\teta)\sol\u\Id,
\end{equation}
where the shear viscosity coefficient $\mu(\teta)$ is a globally Lipschitz function satisfying $ 0<\underline{\mu} (1+\teta) \leq\mu(\teta)$  and the bulk viscosity coefficient $\eta(\teta)$ is a continuous function satisfying $0\leq \eta(\teta) \leq \overline{\eta} (1+ \teta)$. 
The heat flux $\kve$ satisfies Fourier`s law
$\kve(\teta,\nabla\teta)=-\kappa(\teta)\nabla\teta,$ with the heat conductivity coefficient $\kappa(\teta)$, $0<\underline{\kappa}(1+\teta^3)\leq \kappa(\teta)\leq \overline{\kappa}(1+\teta^3) .$ 
The thermodynamical quantities: the pressure $p$, the specific entropy $s$, and the specific internal energy $e$ are specified so that they satisfy the Gibbs relation \begin{equation}\label{Gibbs}\teta Ds(\ro,\teta) = De(\ro,\teta)+ p(\ro,\teta)D(1/\ro).\end{equation} For the purpose of avoiding additional technicalities we assume\footnote{We can also deal with more general constitutive relations for thermodynamical quantities analogous to those from \cite{FeTP}, but it will not bring anything new, except for additional technical complications, thus we do not consider them here.}
\begin{align}
	p(\ro,\teta) =& \ro^\gamma + \ro\teta + \frac{a}{3}\teta^4 ,\\
	e(\ro,\teta) =& \frac{1}{\gamma-1} \ro^{\gamma-1} + c_v\teta + \frac{a}{\ro}\teta^4 ,\label{energ}\\
	s(\ro,\teta) =& \ln\left( \frac{\teta^{c_v}}{\ro}\right) + \frac{4a}{3\ro}\teta^3 ,
\end{align}
where $\gamma$, and $c_v$ are positive constants. The possible values of $\gamma$ will be specified later.
Last, but not least $\sigma$ represents the entropy production rate
$\sigma = \frac{1}{\teta}\left( \S(\teta,\nabla\u):\nabla\u +  \frac{\kve(\teta,\nabla\teta)}{\teta}\cdot\nabla\teta  \right),$ whose non-negativity has to be assumed due to  The Second Law of Thermodynamics.

The fluid is contained in a smooth bounded domain $\Omega$, we assume the following boundary conditions
\begin{gather}
\left.\u\right|_{\partial\Omega}=\vektor{0},\\
	\kve\cdot\en = d(x,\teta)(\teta - \Theta_0), \label{tokh}
\end{gather}
where  $0<\underline{\Theta_0}\leq \Theta_0(x)\in L^{1}(\partial\Omega)$\footnote{We could also allow the function $\Theta_0(x, t)$ to be dependent on time in a time-periodic way, but we omit it.} represents the temperature of the boundary, and for the heat conductivity coefficient $d(x,\teta)$ we will consider two different cases:
 \begin{gather}
		d\text{ dependent on }\teta\text{ satisfying }\nonumber\\
  d_0(1+\teta^3)<d(x,\teta)< C(1+\teta^3),\:d_0>0\label{dep}
\end{gather}
 \begin{gather}
		d\text{ independent of }\teta\text{ satisfying }\nonumber\\
  0<d_0\leq d(x)\leq C<\infty. \label{indep}
\end{gather}
 
 Let us note that according to The Second Law of Thermodynamics, the existence of nontrivial time periodic flow within the energetically closed system is impossible. Hence a condition similar to \refx{tokh} which admits a heat flux through boundary is actually necessary to have the opportunity to get nontrivial solution to our problem.
The fluid is driven by a time-periodic external force $\ef\in L^{\infty}(\R^1\times\Omega,\rtri)$, $\ef(t+L,\cdot)=\ef(t,\cdot) ,\:\forall t\in\R.$

As in the case of the incompressible Navier-Stokes equations\cite{Lera33}, the existence of classical solutions to the system under consideration is far from being obvious. Therefore, we will concern only weak solutions, id est infinite families of integral identities rather than the point-wise satisfaction of equations \refx{RK}--\refx{TE}. The Navier-Stokes-Fourier system enjoys in this context substantial interest of mathematicians; the evolutionary as well as the steady case were studied with various boundary conditions and constitutive relations, see e.g. \cite{Feir04, singular,MuPo09,JeNoPo13}. Recently, the existence of the time-periodic solutions was proved in \cite{FeTP}. However, only the case $\gamma = \frac{5}{3}$ has been treated there, without the radiation on the boundary.\footnote{See also \cite{Axma13} for $\gamma>1$ in the 2D case.}

The structure of the article is as follows. Firstly, we will introduce the concept of the weak solution and present our main result. Secondly, we will show the a-priori estimates for the solutions on purely heuristic level, which is motivation for our definition of weak solutions, and at the same time it is the core of the technical part of the proof, to which the rest of the paper is devoted. We will introduce the approximation scheme, show the existence of approximative solutions and then pass to the limit. Finally, the a-priori bounds in the case without radiation on the boundary allowing one to treat more pressure dependences than in \cite{FeTP}, will be presented in the Appendix.

\subsection{Definition of the solution}

As we search for the time-periodic solutions with period $L$, we will consider all quantities defined on the time interval $\es = \left.[0,L]\right|_{\{0,L\}}]$, accompanied with the periodicity condition $$g(.,0)=g(.,L).$$
   
We call a triple $\{ \ro,\u,\teta \}$ a time-periodic variational entropy solution to the Navier-Stokes-Fourier system, if the following holds true\footnote{Note we are not able to exclude possible vacuum areas.}
\begin{gather}
\ro \geq 0 ,\quad \teta>0 \text{ almost everywhere,} \label{reg}\\
\ro\in \elka{\infty}{\gamma},\:\u\in L^{2}(\es;W^{1,2}_0(\Omega;\R^3)),\\
\teta\in\elka{\infty}{4},\quad\teta^{3/2},\log{\teta} \in L^{2}(\es;\sob)\\
\teta\in L^{4}(\es\times\partial\Omega) , \text{ or }\teta\in L^{1}(\es\times\partial\Omega),\text{ respectively,}
\end{gather}
the renormalized equation of continuity is satisfied for all $b\in C^{1}[0,\infty)$, and any $\psi\in C^{\infty}(\es\times\overline{\Omega})$
\begin{equation}\ins{\inte{ \left(b(\ro)\derit{\psi} + b(\ro)\u\cdot\nabla\psi+\bigl( b(\ro)-b\carka(\ro)\ro \bigr) \sol\u\psi\right) }}=0,\label{rce}
\end{equation}
the momentum equation is satisfied in the sense of distributions, id est for all $\bfi\in C^{\infty}(\es\times{\Omega},\rtri)$ with $\bfi = \vektor{0}$ at $\partial \Omega$
\begin{multline}
	\ins{\inte {  \left(   \ro\u\derit{\bfi}+(\ro\u\otimes\u):\nabla\bfi + p(\ro,\teta)\sol\bfi  \right)   }}\\= 	\ins{\inte {  \left( \S(\teta,\nabla\u):\nabla\bfi - \ro\ef\cdot\bfi  \right) }}, \label{me}
\end{multline}
the specific entropy satisfies for all $\psi\in C^{\infty}(\es\times\overline{\Omega})$
\begin{multline}
	\ins{\inte {  \left(   \ro s(\ro,\teta)\derit{\psi}+\ro s(\ro,\teta)\u\cdot\nabla\psi + \frac{\kve(\teta,\nabla\teta)}{\teta}\cdot\nabla\psi \right)   }}\\= 	\ins{\inth {  \frac{d(\teta-\Theta_0)}{\teta} \psi }} - \left\langle \sigma , \psi\right\rangle,\label{eb}
\end{multline}
where the production of entropy $\sigma$ is represented by a non-negative measure $\sigma\in\mathcal{M}^{+}(\es\times\overline{\Omega})$ satisfying
\begin{equation} \sigma\geq	\frac{1}{\teta}\left( \S(\teta,\nabla\u):\nabla\u -  \frac{\kve(\teta,\nabla\teta)}{\teta}\cdot\nabla\teta  \right), \end{equation}
and the balance of total energy, for all $\psi\in C^{\infty}(\es)$
\begin{equation} 
\begin{array}{c} 
\displaystyle \ins {\biggl(  \derit{\psi}  \inte{\Bigl(  \frac{1}{2}\ro\abs{\u}^2+\ro e(\ro,\teta)  \Bigr) }   \biggr) } \\
\displaystyle =  \ins {{\psi} \Bigl(   \inth { d(\teta- \Theta_0) } - \inte{\ro\ef\cdot\u}   \Bigr) }.\label{te}
\end{array}
\end{equation}

In what follows, we use abridged notation for the norms; more specifically, $\|\cdot\|_{L^p(L^q)}$ means $\|\cdot\|_{L^p(S_1;L^q(\Omega))}$; similarly $\|\cdot\|_{L^p(W^{1,q})}$.  

\section{Main result and a-priori bounds}

\subsection{Main results}

\begin{thm}
Let $\Omega\subset\rtri$ be a bounded domain with a $C^{2+\nu}$ boundary. Assume that the above mentioned hypotheses are all satisfied with 
\begin{align}
		\gamma> & \frac{23}{15}\text{, and }   d \text{ satisfying \refx{indep}.}\hfill
\end{align}
Then for any $M_0>0$ there exists at least one variational entropy time-periodic solution to the Navier-Stokes-Fourier system such that \begin{equation}\inte{\ro(t,\cdot)}=M_0.\label{mass}\end{equation}
\end{thm}

\begin{thm}
Let $\Omega\subset\rtri$ be a bounded domain with a $C^{2+\nu}$ boundary. Assume that the above mentioned hypotheses are all satisfied with 
\begin{align}
		\gamma> & \frac{8}{5}\text{, and }   d \text{ satisfying \refx{dep}.}\hfill
\end{align}
Then for any $M_0>0$ there exists at least one variational entropy time-periodic solution to the Navier-Stokes-Fourier system such that \begin{equation}\inte{\ro(t,\cdot)}=M_0.\label{massa}\end{equation}
\end{thm}

We will present here only the proof of Theorem 1, since the proof of the second one simply copies \cite{FeTP}, except the a-priori bounds given in Appendix.

\begin{notka}
 Note  $\frac{3}{2}<\frac{23}{15}<\frac{8}{5}<\frac{5}{3}$, thus in both cases we deal with more general pressure laws than the known result \cite{FeTP},  and further, in the model with radiation on the boundary we are only $\frac{1}{30}$ above the ``optimal" exponent\footnote{For lower exponents we are not able to bound the kinetic energy in Bogovskii estimates and thus to give a sense to the convective term.} $\gamma =\frac{3}{2}$.
\end{notka}

\begin{note}
 The heat flux $\kve$ satisfies Fourier`s law
$\kve(\teta,\nabla\teta)=-\kappa(\teta)\nabla\teta,$ with the heat conductivity coefficient $\kappa(\teta)$, $0<\underline{\kappa}(1+\teta^3)\leq \kappa(\teta)\leq \overline{\kappa}(1+\teta^3)$ taking into account Stefan-Boltzmann-type radiation, therefore it is natural to take analogous condition also on the boundary. From purely mathematical point of view, the advantage of this choice is that we are be able to deduce better time integrability of the temperature on the boundary $(\teta\in{L^{4}(\es\times\partial\Omega)})$, and consequently also inside the domain due to the Poincar\'{e} inequality. On the other hand we will have to identify the limit for the additional nonlinearity in this model.
\end{note}

\subsection{A-priori bounds in the case of radiation on the boundary}
Before the technical part of the proof we will present formal a-priori estimates, in order to explain the main ideas of the paper.

\subsubsection{Energy estimates}
Our first observation is that the conservation of mass \refx{RK} yields 
\begin{align}
	\ro\in\elka{\infty}{1}. \label{conserv}
\end{align}
Further, we can put $\psi\equiv 1$ in the entropy balance equation to get
\begin{gather}
  \ins{\inte   {\left( \frac{\S(\teta,\nabla \u):\nabla \u}{\teta}   + \frac{\kappa(\teta)\abs{\nabla\teta}^2}{\teta}\right)}}   \nonumber\\+ \ins{\inth {  \frac{d(\teta)\Theta_0}{\teta}  }} \leq \ins{\inth {  d(\teta)}}. 
\end{gather}
Hence, using the form of $d$, $\S$, and the Korn inequality (see e.g. \cite{singular})
\begin{multline}
	\elnor{\u}{L^{2}(\sobo) }^2+ 	\elnorm{\nabla(\teta^{\frac{3}{2}})}{2}{2}^2 + \elnorm{\nabla(\ln\teta)}{2}{2}^2 \\+\elnor{\frac{1}{\teta}}{L^{1}(\es\times\partial\Omega)} +
	\elnor{\teta}{L^{2}(\es\times\partial\Omega)}^2
	\leq C(1+ \elnor{\teta}{L^{3}(\es\times\partial\Omega)}^3) .\label{tempr}
\end{multline}
Integrating the total energy balance over the whole time period, we deduce that
$$\ins { \inth { d(\teta-\Theta_0)}} \leq \ins{ \inte  { \ro \ef\cdot\u} }, $$
id est,
\begin{equation}
\begin{split}
\ins{ \inth  {   (\teta + \teta^4) } }\leq& C \abs{\ins{ \inte  { \ro \ef\cdot\u} }}+C \label{forcing}, \\
\leq& C\left( 1 + \elnorm{\ro}{2}{\frac{6}{5}}  \elnorm{\u}{2}{6} \right).
\end{split}
\end{equation}
Estimating the right-hand side of \refx{tempr} by means of \refx{forcing} we get

\begin{align}
	\elnorm{\teta}{3}{9}^3  + \elnorm{\u}{2}{6}^2\leq C \left( 1 + \elnorm{\ro}{2}{\frac{6}{5}} ^{\frac{6}{5}}  \right),
\end{align}
where we have used the fact that by virtue of the Poincar\'{e} inequality 
\begin{multline*} 
\elnorm{\teta}{3}{9}^3 = \elnorm{\teta^{\frac{3}{2}}}{2}{6}^2 \leq C\left\|\teta^{\frac{3}{2}}\right\|_{L^{2}(\sobo) }^2\\
\leq C\left(\elnor{\teta}{L^{3}(\es\times\partial\Omega)}^3 +	\elnorm{\nabla(\teta^{\frac{3}{2}})}{2}{2}^2  \right).
\end{multline*}
 Further, since interpolation with $\elka{\infty}{1}$ yields
 \begin{align}
	 \elnorm{\ro}{2}{\frac{6}{5}}\leq C \biggl[   1+ \biggl( \ins{\Bigl( \inte{\ro^\gamma} \Bigr)^{\frac{1}{3(\gamma-1)}} }\biggr)^{\frac{1}{2}}   \biggr],
\end{align}
we get
\begin{align}
\elnorm{\teta}{3}{9} &\leq C \biggl[1+ \biggl( \ins{\Bigl( \inte{\ro^\gamma} \Bigr)^{\frac{1}{3(\gamma-1)}} }\biggr)^{\frac{1}{5}}\biggr],\label{tempre} 
	 \\ \elnorm{\u}{2}{6} &\leq C \biggl[   1+ \biggl( \ins{\Bigl( \inte{\ro^\gamma} \Bigr)^{\frac{1}{3(\gamma-1)}} }\biggr)^{\frac{3}{10}}   \biggr].\label{veloc} 
\end{align}

If we denote the total energy by $E(t)=\inte{\bigl(\frac{1}{2} \ro\abs{\u}^2+\ro e(\ro,\teta) \bigr)}$, we get from its balance, with the usage of the structural properties of the internal energy and \refx{forcing} that for all $t_1,\:t_2\in\es$
\begin{align}
	E(t_1)-E(t_2)\leq C\biggl(1+ \ins{ E(t)} \biggr),\nonumber\\
\sup\limits_{t\in\es}E(t)\leq C \biggl(1+ \ins{ E(t)} \biggr).\label{supTE}
\end{align}
From \refx{conserv} and \refx{veloc} we can bound the kinetic energy
  \begin{equation}
  \begin{split}
	  \ins{ \inte{\frac{1}{2} \ro\abs{\u}^2 } } \leq&\:C\elnorm{\u}{2}{6}^2\elnorm{\ro}{\infty}{\frac{3}{2}} \\
	  \leq&\:C  \biggl[   1+ \biggl( \ins{\Bigl( \inte{\ro^\gamma} \Bigr)^{\frac{1}{3(\gamma-1)}} }\biggr)^{\frac{3}{5}} \elnorm{\ro}{\infty}{\gamma}^{\frac{\gamma}{3(\gamma-1)}} \biggr] \\
	  \leq&\:C \left(1+\sup\limits_{t\in\es}E(t)^{ \frac{1}{5(\gamma-1)}+\frac{1}{3(\gamma-1)} }\right)	 ,
	  \end{split}
	\end{equation}
	since $ \frac{1}{5(\gamma-1)}+\frac{1}{3(\gamma-1)}<1$ for $\gamma>\frac{23}{15}$ we can absorb the term on the right-hand side using Young's inequality.	Hence we get for the total energy
\begin{align}	\sup\limits_{t\in\es}E(t)\leq &\:C \biggl(1+  \ins{ \inte { \ro e(\ro,\teta) }  }  \biggr)\\
\leq &  C \biggl(1+  \ins{ \inte {  (\ro^\gamma + \ro\teta + \teta^4)    }  }  \biggr).
\end{align}
The first term on the right-hand side will be left as it is. The last term will be estimated using \refx{tempre} as follows
\begin{equation}
\elnorm{\teta}{4}{4}^4\leq \elnorm{\teta}{3}{9}^3\elnorm{\teta}{\infty}{4} \leq C\left(1+\elnorm{\ro}{\gamma}{\gamma}^\frac{\gamma}{5(\gamma-1)} \sup\limits_{t\in\es}E^{1/4}(t)\right) ,\label{teticka}
\end{equation}
which can be absorbed to the left-hand side by means of Young's inequality, since $\frac{4}{3}\cdot\frac{\gamma}{5(\gamma-1)}<\gamma$ for our range of $\gamma$'s.
Similarly, we have
\begin{equation}
\begin{split}
	\ins{ \inte{ \ro\teta } } \leq& \elnorm{\ro}{\frac{3}{2}}{\frac{9}{8}} \elnorm{\teta}{3}{9} \\ 
	\leq&  C\left( 1+ \elnorm{\ro}{\gamma}{\gamma}^{\frac{\gamma}{9(\gamma-1)}+\frac{\gamma}{15(\gamma-1)} } \right)\\ 
	\leq&  C\left( 1+ \elnorm{\ro}{\gamma}{\gamma}^{\frac{8\gamma}{45(\gamma-1)} } \right),\label{strycek}
	\end{split}
\end{equation}
thus using the estimates above
\begin{align}
	\sup\limits_{t\in\es}E(t)\leq C \biggl(1+  \elnorm{\ro}{{\gamma}}{{{\gamma}}}^{\gamma} \biggr).\label{super}
\end{align}

\subsubsection{Pressure estimates}
It remains to deduce suitable estimates of density, which will be done by testing the momentum equation by\footnote{We use notation   $\left\{g\right\}_{\Omega}=\frac{1}{\abs{\Omega}}\int_{\Omega}{g}\:\dx $. }
$$\Phib = \B\left[  \ro^{\gamma(a-1)}-\left\{ \ro^{\gamma(a-1)}\right\}_{\Omega}  \right],$$
where $a>1$ will be specified later, and $\B \sim \sol^{-1}$ is the Bogovskii operator (see e.g. \cite{singular}). We assume $\gamma(a-1)\leq 1$; due to the properties of the Bogovskii operator it follows that $\Phib\in L^{\infty}\left( \es;W^{\frac{1}{\gamma(a-1)}}(\Omega)\right),$ and $\left\{ \ro^{\gamma(a-1)}\right\}_{\Omega} \in L^{\infty}(\es)$, so we obtain

\begin{multline*}
	\ins{\inte {  \left(   p(\ro,\teta) \ro^{\gamma(a-1)}  \right)   }}\\
	\leq	\ins{\inte {  \left( -\ro\u\derit{\Phib}-(\ro\u\otimes\u):\nabla\Phib+ \S(\teta,\nabla\u):\nabla\Phib - \ro\ef\cdot\Phib  \right) }}\\
	 +C\ins{\inte {  \left((  \ro^{\gamma}+\ro\teta   +\teta^{4}  ) \left\{\ro^{\gamma(a-1)}\right\}_\Omega \right) }}. 
	\end{multline*}
	
The terms on the left-hand side of the inequality have sign and give the desired estimate of $\ro^{a\gamma}$, if the right-hand side will be estimated. We will present only the most difficult and restrictive terms. We start with the convective term, since it determines the possible values of $a$ and consequently also of $\gamma$.
\begin{equation}
\begin{split}
	&\ins{\inte { (\ro\u\otimes\u):\nabla\Phib} }\leq \elnorm{\u}{2}{6}^2\elnorm{\ro\nabla\Phib}{\infty}{{3}/{2}}\\ &\qqquad\quad  \leq 
	C\elnor{\ro}{L^{a\gamma}(L^{a\gamma})}^{\frac{a\gamma}{5(a\gamma-1)}} \elnorm{\ro}{\infty}{p}\elnorm{\nabla\Phib}{\infty}{\frac{p}{\gamma(a-1)}}\\
&\qqquad\quad  \leq \frac{1}{14}\elnor{\ro}{L^{a\gamma}(L^{a\gamma})}^{a\gamma} + C
	\left( \elnorm{\ro}{\infty}{p}\elnorm{\ro^{\gamma(a-1)}}{\infty}{\frac{p}{\gamma(a-1)}}\right)^{\frac{5a\gamma-5}{5a\gamma-6}},\label{domu}
	\end{split}
\end{equation}
where we have used an analogy of the estimate \refx{veloc} with $a\gamma$ instead of $\gamma$, the properties of the Bogovskii operator;  $p$  satisfies $\frac{2}{3} = \frac{1}{p}+\frac{\gamma(a-1)}{p}$, yielding $p=\frac{3}{2}\cdot\bigl(1+\gamma(a-1)\bigr)$.

Further, $\elnorm{\ro}{\infty}{p}\elnorm{\ro^{\gamma(a-1)}}{\infty}{\frac{p}{\gamma(a-1)}} = \elnorm{\ro}{\infty}{p}^{1+\gamma(a-1)},$
 and we would like to interpolate as follows
 $$  \elnorm{\ro}{\infty}{p} \leq \elnorm{\ro}{\infty}{\gamma}^\alpha\elnorm{\ro}{\infty}{1}^{1-\alpha}, $$
 therefore we need $p\in(1,\gamma)$, which leads to the first constraint on the possible values of $a$, namely \makebox{${a<\frac{5\gamma-3}{3\gamma};}$} if this condition is satisfied we get 
 $\alpha = \dfrac{\gamma}{\gamma-1}\cdot\dfrac{3\gamma a-3\gamma+1}{3\gamma a -3\gamma+3}.$ Thus, using \refx{conserv}
 \begin{equation}
\begin{split}
	&\ins{\inte { (\ro\u\otimes\u):\nabla\Phib} } \\
&\qqquad\quad  \leq \frac{1}{14}\elnor{\ro}{L^{a\gamma}(\es\times\Omega)}^{a\gamma} + C
	\left( \elnorm{\ro}{\infty}{p}^{\bigl( 1+\gamma(a-1) \bigr) } \right)^{\frac{5a\gamma-5}{5a\gamma-6}}\\
&\qqquad\quad  \leq \frac{1}{14}\elnor{\ro}{L^{a\gamma}(\es\times\Omega)}^{a\gamma} + C
	\left( \elnorm{\ro}{\infty}{\gamma}^{\alpha \cdot \bigl( 1+\gamma(a-1) \bigr) } \right)^{\frac{5a\gamma-5}{5a\gamma-6}}.
	\end{split}
\end{equation}
 The first term can be immediately push to the left-hand side, while for the second one, we need
 $$\frac{1}{\gamma}\cdot \alpha \cdot \bigl( 1+\gamma(a-1) \bigr)   \cdot{\frac{5a\gamma-5}{5a\gamma-6}}<1,$$
 which leads to the quadratic inequality for the quantity $A:=a\gamma$
$$
 (5a\gamma-5)\cdot(3a\gamma-3\gamma+1)<3(\gamma-1)\cdot(5a\gamma-6).
$$
Denoting $A= a \gamma$, we have
$$
15A^2+A(5-30\gamma) + 33\gamma-23<0,
$$
 the discriminant $D_A =  5\cdot(180\gamma^2-456\gamma+281) $ is definitely positive for all $\gamma>\frac{3}{2}$, so we have to ensure that 
 \begin{align}
  \frac{30\gamma-5-\sqrt{D_A}}{30}&<A<\frac{30\gamma-5+\sqrt{D_A}}{30}\text{, which means}\\
1 &<a<1+\frac{-5+\sqrt{D_A}}{30\gamma},
\end{align}
 since we consider only $a>1$. Therefore we need 
 \begin{align*}
 -5+\sqrt{D_A}>&0,\qquad \mbox{id est,}\\
 180\gamma^2-456\gamma+276>&0
 \end{align*}
 which yields again the restriction \begin{equation} \gamma>\frac{23}{15}.\label{restriction}\end{equation}
 Conversely, we are able to choose for each $\gamma$ satisfying \refx{restriction}, $a>1$ such that we can bound the convective term, namely\footnote{The first quantity is less than the second one for $\gamma>\frac{39}{25}$.}
 \begin{equation}1<a<\min\left\{\frac{5\gamma-3}{3\gamma} ,  1+\frac{-5+\sqrt{5\cdot(180\gamma^2-456\gamma+281)}}{30\gamma}, \frac{\gamma+1}{\gamma}   \right\}.\label{posvala}\end{equation}
 
 For the term with $\derit{\Phib}$, we will use the renormalized equation of continuity with $b(\ro) = \ro^{\gamma(a-1)}$; we obtain two terms, one could be estimated similarly as above, the other as follows (note that $\frac{6p}{7p-6} = \frac{2p}{p+2\gamma(a-1)}$ for $p = \frac{3}{2} (1+ \gamma(a-1))$)
\begin{equation*}
\begin{split}
&\ins{ \inte { \ro\u\B \left[ \ro^{\gamma(a-1)}\sol\u -\left\{\ro^{\gamma(a-1)}\sol\u \right\}_\Omega \right] } } 
 \\&\quad\quad\quad\:\leq \elnorm{\ro\u}{2}{\frac{6p}{p+6}} 
\elnorm{ \B \left[ \ro^{\gamma(a-1)}\sol\u -\left\{\ro^{\gamma(a-1)}\sol\u \right\}_\Omega \right]  } {2}{\frac{6p}{5p-6}}
\\&\quad\quad\quad\:\leq C\elnorm{\ro}{\infty}{p} \elnorm{\u}{2}{6} \elnorm{\ro^{\gamma(a-1)}\sol\u } {2}{\frac{2p}{p+2\gamma(a-1)}}
\\&\quad\quad\quad\:\leq C\elnorm{\ro}{\infty}{p} \elnorm{\u}{2}{6} \elnorm{\ro^{\gamma(a-1)} } {\infty}{\frac{p}{\gamma(a-1)}}
 \elnorm{\sol\u}{2}{2},
\end{split}
\end{equation*}
with the same $p$ as above.
 The right-hand side can be estimated again in the same way as in \refx{domu}.

While estimating the terms with the temperature, we will exploit inequality \refx{tempre}, which stems from presence of the radiation on the boundary. Similarly, as in \refx{teticka} and \refx{strycek}
\begin{equation}
\begin{split}
&\ins{\inte {  \left(\teta^{4}  \left\{\ro^{\gamma(a-1)}\right\}_\Omega \right) }}+	\ins{\inte { \ro\teta   \left\{\ro^{\gamma(a-1)}\right\}_\Omega  }}\\
	&\qqquad\leq C \left(1+\elnorm{\teta}{3}{9}^3\elnorm{\teta}{\infty}{4} + \elnorm{\ro}{\frac{3}{2}}{\frac{9}{8}}\elnorm{\teta}{3}{9}\right) \\
	&\qqquad\leq C \left(1+\sup\limits_{t\in\es}E(t)^{\frac{1}{4}+\frac{1}{5(\gamma-1)}} +  \sup\limits_{t\in\es}E(t)^{\frac{1}{9(\gamma-1)}+\frac{1}{15(\gamma-1)}}   \right),	
	\end{split}
\end{equation}
where the powers of the energy are less than one.
Finally
$$ \abs{\ins{\inte{ \ro\ef\cdot\u }}} \leq C \elnorm{\sqrt{\ro}\u}{\infty}{2} \elnorm{\sqrt{\ro}}{\infty}{2}\leq C\left(1+\sup\limits_{t\in\es}E(t)^{\frac{1}{2}}\right).$$

Thus, we get
$$  \ins{\inte {  \ro^{a\gamma}  }} \leq C \left(1+\sup\limits_{t\in\es}E(t)^{\beta}\right),$$
with some $\beta<1$. This estimate can be plugged into \refx{super} in order to get the bound
\begin{equation}\sup\limits_{t\in\es}E(t)<\infty.\label{superenergy}
\end{equation}
Moreover, due to the obtained estimates we can derive higher integrability of the temperature on the boundary. Namely, from \refx{tempr} and  \refx{superenergy} we have $$\teta^{\frac{3}{2}}\in L^{2}(\es;\sob),\qquad \teta\in\elka{\infty}{4} $$
and we can interpolate
\begin{align}
\elnor{\teta}{L^{13/3}(\partial\Omega)}^\frac{13}{3}=&\inth{\teta^{\frac{3}{2}\cdot\frac{26}{9}}}\leq C\elnor{\teta^{\frac{3}{2}}}{\sob}\left(\inte{ \teta^{\frac{3}{2}\cdot2\cdot(\frac{26}{9}-1)}}  \right)^{\frac{1}{2}}\\
\leq &  C\elnor{\teta^{\frac{3}{2}}}{\sob} \elnor{\teta}{L^{\frac{17}{3}}(\Omega)}^{\frac{17}{6}}
\leq   C\elnor{\teta^{\frac{3}{2}}}{\sob} \elnor{\teta}{L^{9}(\Omega)}^{\frac{3}{2}} \elnor{\teta}{L^{4}(\Omega)}^{\frac{4}{3}},\label{hranic}
\end{align}
\begin{align}
\elnor{\teta}{L^{13/3}(\es\times\partial\Omega)}^\frac{13}{3}\leq C \elnor{\teta^{\frac{3}{2}}}{L^{2}(\sob)}\elnorm{\teta}{3}{9}^{\frac{3}{2}} \elnorm{\teta}{\infty}{4}^{\frac{4}{3}}\leq C. \label{hranice}
\end{align}
Now, we are ready to start the proof of our main theorem in the case of the radiation on the boundary.


\section{Approximation}

\subsection{Approximation scheme}

Following \cite{FeTP}, we will approximate the original problem introducing five parameters, namely $N\in\N$ representing the dimension of the finite dimensional space for the velocity field, $\tau>0$, and $\ksi>0$\footnote{We will finally set $\ksi=\delta$. However, we keep this notation for the purpose of clarity.} introduced in order to get an information about the time integrability of the velocity, and temperature, respectively, even in the possible vacuum zones, $\epsil>0$ representing the parabolic regularization of the continuity equation, and last, but not least $\delta>0$ regularizing the pressure and heat flux in order to get higher integrability of the density and the temperature. Moreover, $\Gamma$ and $B$ denote sufficiently large positive numbers. We will search for $\ro\geq 0$, $\ro\in C^{\infty}\left(\es;\:\sobdp\right)$, $\u_N$ in some finite dimensional space, and $\teta>0$ such that $\ln\teta$, $\teta\in\sobdps$ for any $p<\infty$. Hence, we replace the original system by the following approximative version.

We add artificial diffusion and mass production into the continuity equation, and add a boundary condition in order to conserve the mass, denoting $m=\frac{M_0}{\abs{\Omega}}$
\begin{equation}
\begin{split}
\derit{\ro} + \sol(\ro \u_N) -\epsil \Delta \ro + \epsil \ro= \epsil m  \qquad \text{ in }\es\times\Omega, \label{CEf} \\
\frac{\partial \ro}{\partial \en} = 0 \qquad \text{ on }\es\times\partial\Omega.
\end{split}
\end{equation}

We modify the pressure and consider the Galerkin approximation in the momentum equation. For this purpose we introduce a finite-dimensional subspaces of $L^{2}(\es;\sobo)$ with basis, consisting of $\wi(t,x) = a^k(t)\vektor{b}^l(x) $, with $i=i(k,l)$, $i=1,\ldots,N$, which is orthonormal with respect to scalar product $\left(\wi,\wj\right) =\ins{\inte{\nabla\wi:\nabla\wj}}$. Here $a^k$ stands for orthonormal basis of goniometric functions, which are smooth and $L$-periodic,\footnote{For example $\cos\left(\frac{\pi (k+1)x}{L}\right)$, and $\sin\left(\frac{\pi kx}{L}\right)$ for $k$ odd, or even, respectively.} while $\vektor{b}^l$ forms orthonormal basis of $\sobo$, such that all its elements belong to $W^{2,p}(\Omega)$ for any $p<\infty$.

\begin{multline}
\ins{\inte {  \left(   \ksi\derit{\u_N} \cdot\wi+\derit{(\ro\u_N)}\cdot\wi-(\ro\u_N\otimes\u_N):\nabla\wi   \right)   }} \\
+ \ins{\inte {  \Bigl( \S(\teta,\nabla\u_N):\nabla\wi- \left(p(\ro,\teta) +\delta (\ro^\Gamma + \ro^2) \right)\sol\wi \Bigr)}}\label{MEf} \\
= 	\ins{\inte { \left( -\epsil\nabla\ro \cdot\nabla\u_N\wi + \frac{1}{2}\epsil(m-\ro)\u_N\cdot \wi +\ro\ef\cdot\wi  \right) }} 
\end{multline}

We transform the regularized internal energy equation, which we see as an equation for the temperature, by means of the so called Kirchhoff transform
\begin{equation}\label{KirT}
\Phi(g)=\displaystyle\int\limits_0^g  \bigl(\kappa(\ex{z})\ex{z} + \delta \ex{(B+1)z} + \delta \bigr) \de{z}.
\end{equation}
Note that, since the integrand of the integral is measurable and greater than $\delta$, $\Phi$ is continuous, increasing, and one-to-one with $\Phi^{-1}$ Lipschitz continuous, in particular having a linear growth. We get
 \begin{equation}
 \begin{split}
 -\tau \deritt{\Phi(\ln\teta)} + \ksi\derit{\teta}+ \tau \Phi(\ln\teta) 
 + \derit{(\ro e)}-\sol\nabla\Phi(\ln\teta)+ \sol(\ro e\u_N)
  \\=\S(\teta,\nabla\u_N):\nabla \u_N - p(\ro,\teta)\sol\u_N + \epsil\delta(\Gamma \ro^{\Gamma-2}+2)\abs{\nabla\ro}^2 + \delta\teta^{-1} 
\\\text{ in }\es\times\Omega ,\label{ENf}\\
  \bigl(\kappa(\teta) + \delta\teta^{B}+\delta\teta^{-1}\bigr)\frac{\partial\teta}{\partial\en}=d(x,\teta)(\Theta_0-\teta)\qquad\text{ on }\es\times\partial\Omega.
  \end{split}
 \end{equation}
Since we have in our definition of the solution  the entropy equation instead of the energy equation we will present now also its approximative version. It could be derived by dividing \refx{ENf} by temperature $\teta$ with usage of \refx{CEf}
\begin{multline}
-\tau\derits{\frac{\Phi\carka(\ln\teta)}{\teta} \derits{\ln\teta}} - \tau\frac{\Phi\carka(\ln\teta)}{\teta^3}\left(\derit{\teta}\right)^2 + \ksi\derit{\ln\teta} + \derit{(\ro s)} \\
+ \tau\frac{\Phi(\ln\teta)}{\teta} + \bigl(\sol(\ro\u_N) +\det{\ro} \bigr)\frac{\ro e + p - \ro\teta s}{\ro\teta}
+\sol(\ro s \u_N) \\
- \sol \Bigl( \bigl(\kappa(\teta) + \delta\teta^{B}+\delta\teta^{-1}\bigr)\frac{\nabla\teta}{\teta}\Bigr) = \frac{1}{\teta}\S(\teta,\nabla\u_N):\nabla \u_N \\
+ \bigl(\kappa(\teta) + \delta\teta^{B}+\delta\teta^{-1}\bigr)\frac{\abs{\nabla\teta}^2}{\teta^2}+\delta\frac{1}{\teta^2} + \frac{\epsil\delta}{\teta}(\Gamma\ro^{\Gamma-2}+2)\abs{\nabla\ro}^2 \\ \text{ in }\es\times\Omega.\label{ETf}
\end{multline}

\subsection{Existence of approximate solutions for fixed parameters}
The main aim of this subsection is to show the following existence result.
\begin{lem}\label{exist}
For an arbitrary fixed set of parameters $N\in\N$, $\tau$, $\ksi$, $\epsil$, $\delta>0$ such that $\epsil \ll \delta$, there exists  at least one solution to the approximate scheme, id est $\ro\geq 0$, $\ro\in C^{\infty}\left(\es,\:\sobdp\right)$, $\u_N\in \LwiN$, and $\teta>0$, $\ln\teta$, $\teta\in\sobdps$,  such that \refx{CEf}, \refx{MEf}, and \refx{ENf} hold.
\end{lem}

The main idea of the proof is similar to the proof presented in \cite{FeTP}; we repeat the main steps for the sake of clarity.
First of all, we observe that as soon as we have the velocity field, we are able to recover the density, namely
\begin{prop} \label{CEfx}
For any velocity field $\unt\in \LwiN$, there exists a density $\ro\in C^{\infty}\left(\es;\sobdp\right)$ satisfying $\ro \geq 0$, 
\begin{equation}
\begin{split}
\det\ro + \sol(\ro\unt)-\epsil\Delta\ro +\epsil\ro=\epsil m\text{ in }\es\times\Omega\\
\frac{\partial \ro}{\partial \en} = 0 \text{ on }\es\times\partial\Omega. \label{contfx}
\end{split}
\end{equation}
Moreover,
\begin{equation} \label{59a}
\int_\Omega \ro \dx = m |\Omega| = M_0.
\end{equation}
\end{prop}

The proof is standard and it is based on application of a version of the Schauder fixed point theorem formulated below. Note \eqref{59a} is a direct consequence of \eqref{contfx} integrated over $\Omega$ and the uniqueness argument for ordinary differential equations.
 

We have

\begin{LSFPT}
Let $X$ be a Banach space, and $\Te$ a continuous and compact mapping $\Te:\:X \mapsto X$, such that the possible fixed points
 $x = \lambda \Te x$, $0 \leq \lambda \leq 1$
are bounded in $X$. Then $\Te$ possesses a fixed point.
\end{LSFPT}

We will apply this theorem on the mapping $$\Te: \LwiN\times \soblps\mapsto \LwiN\times \soblps$$
which is defined as a solving operator  to the following system $(\Te(\unt,\ln\tteta)=\Te(\u_N,\ln\teta))$:
\begin{multline}
\INT{ \left(   \ksi\derit{\u_N} \cdot\wi+\derit{(\ro\unt)}\cdot\wi-(\ro\unt\otimes\unt):\nabla\wi   \right)   } \\
+ \INT{  \Bigl( \S(\tteta,\nabla\u_N):\nabla\wi- \left(p(\ro,\tteta) +\delta (\ro^\Gamma + \ro^2) \right)\sol\wi \Bigr)}\label{MEfx} \\
= 	\INT { \biggl( -\epsil\nabla\ro \cdot\nabla\unt\wi + \frac{1}{2}\epsil(m-\ro)\unt\cdot\wi +\ro\ef\cdot\wi  \biggr) } \\
i=1,\ldots,N
\end{multline}
 \begin{equation}
 \begin{split}
 -\tau \deritt{\Phi(\ln\teta)} + \ksi\derit{\tteta}+ \tau \Phi(\ln\teta) 
 + \derit{(\ro \te)}-\sol\nabla\Phi(\ln\teta)+ \sol(\ro \te\unt)
  \\=\S(\tteta,\nabla\unt):\nabla \unt - p(\ro,\tteta)\sol\unt + \epsil\delta(\Gamma \ro^{\Gamma-2}+2)\abs{\nabla\ro}^2 + \delta\tteta^{-1} 
\\\text{ in }\es\times\Omega ,\label{ENfx}\\
 \frac{\partial\Phi(\ln\teta)}{\partial\en}=d(x,\tteta)(\Theta_0-\tteta),\qquad\text{ on }\es\times\partial\Omega,
  \end{split}
 \end{equation}
 where $\ro$ is defined as a solution to \refx{contfx} from Proposition \ref{CEfx}, $\Phi$ is as above. We have intorduced the notation $\te=e(\ro,\tteta)$.
 
 Considering the momentum equation, it is easy to show the existence of solution to the corresponding system of linear algebraic equations, using Korn's inequality and Brower fixed point theorem.
 \begin{prop}
  For any $\unt\in\LwiN$, $\tteta\in\soblps$, and a corresponding $\ro\in \soblps$ from Proposition \ref{CEfx}, there exists a unique solution to $\refx{MEfx}$. Furthermore, it satisfies $\u_N\in \cinf\left(\es\times\bar{\Omega}, \LwiN\right).$
\end{prop}

The second part of the solving operator $\Te$ is defined through the energy equation, 
\begin{prop}\label{EEfx}
For any $\unt\in\LwiN$,$\tteta\in\soblps$, and a corresponding $\ro$ from Proposition \ref{CEfx}, there exists a uniquely defined $\teta>0$ such that $\teta$ and $\ln\teta$ $\in\sobdps$, satisfying \refx{ENfx}.
\end{prop}
\begin{proof}[Proof of Proposition \ref{EEfx}]
The crucial point in the proof is that instead of searching directly for $\teta$, we solve the system for $\ln\teta$, and then set  $\teta :=\ex{\ln\teta}$, which immediately implies $\teta>0.$ More precisely, we solve the elliptic problem for $Z$
 \begin{multline*}
 -\tau \deritt{Z} + \ksi\derit{\tteta}+ \tau Z 
 + \derit{(\ro \te)}-\sol\nabla Z+ \sol(\ro \te\unt) -\delta\tteta^{-1}  \\
 =\S(\tteta,\nabla\unt):\nabla \unt - p(\ro,\tteta)\sol\unt + \epsil\delta(\Gamma \ro^{\Gamma-2}+2)\abs{\nabla\ro}^2 \text{ in }\es\times\Omega ,
  \end{multline*}
  
  \vspace{-0.8cm}
  
  \begin{align}\label{ENZfx}
 \frac{\partial Z}{\partial\en}=d(x,\tteta)(\Theta_0-\tteta)\qquad\text{ on }\es\times\partial\Omega,
  \end{align}
and then define $\ln\teta=\Phi^{-1}(Z),$ which is well-defined thanks to \refx{KirT} and the note below it.
\end{proof}

To summarize, $\Te$ is a compact continuous operator from $\LwiN\times\soblps$ into itself. Thus, it remains to prove the boundedness of the possible fixed points \begin{equation}\lambda \Te(\u_N,\teta) = (\u_N,\teta)\text{, for }0\leq\lambda\leq1\label{fixed}\end{equation} in the space $\LwiN\times\soblps.$ Formula \refx{fixed} is  nothing but

\begin{multline}
\ins{\inte {  \left(   \ksi\derit{\u_N} \cdot\wi+\lambda\derit{(\ro\u_N)}\cdot\wi-\lambda(\ro\u_N\otimes\u_N):\nabla\wi   \right)   }} \\
+ \ins{\inte {  \Bigl( \S(\teta,\nabla\u_N):\nabla\wi- \lambda\left(p(\ro,\teta) +\delta (\ro^\Gamma + \ro^2) \right)\sol\wi \Bigr)}}\label{MEfl} \\
= \lambda	\ins{\inte { \left( -\epsil\nabla\ro \cdot\nabla\u_N\wi + \frac{1}{2}\epsil(m-\ro)\u_N\cdot \wi +\ro\ef\cdot\wi  \right) }} 
\end{multline}

\begin{equation}
 \begin{split}
 -\tau \deritt{\Phi(\ln\teta)} + \lambda \ksi\derit{\teta}+ \tau \Phi(\ln\teta) 
 + \lambda\derit{(\ro e)}-\sol\nabla\Phi(\ln\teta)+ \lambda\sol(\ro e\u_N)
  \\=\lambda\S(\teta,\nabla\u_N):\nabla \u_N - \lambda p(\ro,\teta)\sol\u_N + \lambda\epsil\delta(\Gamma \ro^{\Gamma-2}+2)\abs{\nabla\ro}^2 + \lambda\delta\teta^{-1} 
\\\text{ in }\es\times\Omega ,\label{ENfl}\\
  \bigl(\kappa(\teta) + \delta\teta^{B}+\delta\teta^{-1}\bigr)\frac{\partial\teta}{\partial\en}=\lambda d(x,\teta)(\Theta_0-\teta)\qquad\text{ on }\es\times\partial\Omega,
  \end{split}
 \end{equation} where $\ro$ satisfies \refx{CEf}, and $\Phi$ is given by \refx{KirT}.
Using $\u_N$ as a test function in \refx{MEfl} with help of integration by parts yields
\begin{multline}
\ins{\inte {  \left(   \ksi\derit{\u_N} \cdot\u_N+\lambda\derit{(\ro\u_N)}\cdot\u_N-\lambda(\ro\u_N\otimes\u_N):\nabla\u_N   \right)   }} \\
+ \ins{\inte {  \Bigl( \S(\teta,\nabla\u_N):\nabla\u_N- \lambda\left(p(\ro,\teta) +\delta (\ro^\Gamma + \ro^2) \right)\sol\u_N \Bigr)}} \\
= \lambda	\ins{\inte { \left( \frac{1}{2}\epsil\Delta\ro \abs{\u_N}^2 + \frac{1}{2}\epsil(m-\ro)\abs{\u_N}^2 +\ro\ef\cdot\u_N \right) }} .
\end{multline}
Using \refx{CEf} multiplied on $\lambda\frac{1}{2} \abs{\u_N}^2$ with another integration by parts\footnote{The terms in the form of time derivative vanish due to the time periodic condition.} we get
\begin{multline}
 \ins{\inte {  \Bigl( \S(\teta,\nabla\u_N):\nabla\u_N\Bigr) }} \\
 =  \ins{\inte {  \Bigl(\lambda\left(p(\ro,\teta) +\delta (\ro^\Gamma + \ro^2) \right)\sol\u_N  +\ro\ef\cdot\u_N \Bigr) }} \label{uuu}.
\end{multline}
Integrating the energy equation \refx{ENfl} over $\es\times\Omega$ we obtain with usage of the boundary condition
\begin{multline}
\tau \ins{\inte{ \Phi(\ln\teta) }} +\lambda\ins{\inth{ d(x,\teta)\teta }}   \\
=\lambda\ins{\inte{ \Bigl(\S(\teta,\nabla\u_N):\nabla \u_N - p(\ro,\teta)\sol\u_N + \delta\teta^{-1}\Bigr)}} \\
+ \lambda\epsil\delta\ins{\inte{(\Gamma \ro^{\Gamma-2}+2)\abs{\nabla\ro}^2 }}+\lambda\ins{\inth{ d(x,\teta)\Theta_0 }}.\label{intENfl}
\end{multline}
Further, we get renormalized version of the continuity equation by multliplying \refx{CEf} by $\frac{\beta}{\beta-1}\ro^{\beta-1}$ after some obvious manipulations
\begin{multline}
\epsil\beta\INT{ \Bigl(\frac{1}{\beta-1}\ro^\beta+\ro^{\beta-2}\abs{\nabla\ro}^2  \Bigr) }\\
+\INT{\ro^\beta \sol\u_N }=\epsil \frac{\beta}{\beta-1}\INT{m\ro^{\beta-1} }.\label{RCEf}
\end{multline}
In order to get the  total energy balance, we sum up \refx{uuu}, \refx{intENfl} and \refx{RCEf} with $\beta=2, \Gamma$ multiplied by $\delta\lambda$. This reads
\begin{multline}
(1-\lambda) \ins{\inte {  \bigl( \S(\teta,\nabla\u_N):\nabla\u_N\bigr) }}+\tau \ins{\inte{ \Phi(\ln\teta) }} \\
+\lambda\ins{\inth{ d(x,\teta)\teta }} +\epsil\delta\lambda\INT{\Bigl(\frac{\Gamma}{\Gamma-1}\ro^{\Gamma}+2\ro^{2}\Bigr)}\\
=\lambda\INT{\Bigl(\ro\ef\cdot\u_N + \delta\teta^{-1} \Bigr) }+\lambda\ins{\inth{ d(x,\teta)\Theta_0 }}
\\+\lambda\INT{\Bigl(\epsil\delta\frac{\Gamma}{\Gamma-1}m\ro^{\Gamma-1}+2\epsil\delta m \ro \Bigr) }.
\end{multline}
The last two integrals on the right-hand side can be pushed to the left-hand side by means of Young's inequality obtaining
\begin{multline}
(1-\lambda) \ins{\inte {  \bigl( \S(\teta,\nabla\u_N):\nabla\u_N\bigr) }}+\tau \ins{\inte{ \Phi(\ln\teta) }} \\
+\lambda\ins{\inth{ d(x,\teta)\teta }} +\epsil\delta\lambda\INT{\Bigl(\frac{\Gamma}{\Gamma-1}\ro^{\Gamma}+2\ro^{2}\Bigr)}\\
\leq  C \Bigl(1+\lambda\INT{\bigl(\ro\ef\cdot\u_N + \delta\teta^{-1} \bigr) }\Bigr). \label{TEBfl}
\end{multline}
By similar arguments which lead from \refx{ENf} to \refx{ETf} we can obtain from \refx{ENfl} the following entropy identity
\begin{multline}
-\tau\derits{\frac{\Phi\carka(\ln\teta)}{\teta} \derits{\ln\teta}} - \tau\frac{\Phi\carka(\ln\teta)}{\teta^3}\left(\derit{\teta}\right)^2 +\lambda \ksi\derit{\ln\teta} +\lambda \derit{(\ro s)} \\
+ \tau\frac{\Phi(\ln\teta)}{\teta} + \lambda\bigl(\sol(\ro\u_N) +\det{\ro} \bigr)\frac{\ro e + p - \ro\teta s}{\ro\teta}
+\sol(\ro s \u_N) \\
- \sol \Bigl( \bigl(\kappa(\teta) + \delta\teta^{B}+\delta\teta^{-1}\bigr)\frac{\nabla\teta}{\teta}\Bigr) = \lambda\frac{1}{\teta}\S(\teta,\nabla\u_N):\nabla \u_N \\
+ \bigl(\kappa(\teta) + \delta\teta^{B}+\delta\teta^{-1}\bigr)\frac{\abs{\nabla\teta}^2}{\teta^2}+\lambda\delta\frac{1}{\teta^2} +\lambda \frac{\epsil\delta}{\teta}(\Gamma\ro^{\Gamma-2}+2)\abs{\nabla\ro}^2 \\ \text{ in }\es\times\Omega,\label{ETfl}
\end{multline}which can be integrated over $\es\times\Omega$ yielding
\begin{equation}\begin{split}
&\INT{\bigl(\kappa(\teta)+\delta \teta^B+\delta \teta^{-1}\bigr) \frac{\abs{\nabla\teta}^{2}}{\teta^2} }+ \tau\INT{ \Phi\carka(\ln\teta)\frac{\left(\det{\teta}\right)^2}{\teta^3} } \\
&\qquad +\lambda\ins{ \inth{\frac{1}{\teta}d(x,\teta)\Theta_0}} +\lambda\INT{  \Bigl(\frac{1}{\teta} \S(\teta,\nabla\u_N):\nabla\u_N + \frac{\delta}{\teta^2}\Bigr)} \\
&\qquad +\epsil\delta\lambda\INT {  \frac{1}{\teta}(\Gamma\ro^{\Gamma-2}+2)\abs{\nabla\ro}^2 }  =\lambda\ins{ \inth{ d(x,\teta)\Theta_0  } }\\
&\qquad\quad+\tau\INT{ \frac{\Phi(\ln\teta)}{\teta}  }+\INT{\bigl(\sol(\ro\u_N) +\det{\ro} \bigr)\frac{\ro e + p - \ro\teta s}{\ro\teta}}.\label{intETfl}
\end{split}\end{equation}
In order to estimate the last term on the right-hand side, we will use the continuity equation \refx{CEf}. Hence the last integral transforms into
\begin{equation}\epsil\INT{\bigl(  m-\ro+\Delta\ro  \bigr)\frac{\ro e + p - \ro\teta s}{\ro\teta}},\label{epscast}\end{equation}
and by virtue of Gibbs' relation \refx{Gibbs} we recognize the terms with the negative sign
\begin{equation}
\begin{split}
&\epsil\INT{\Delta \ro  \frac{\ro e(\ro,\teta) + p(\ro,\teta) - \ro\teta s(\ro,\teta)}{\ro\teta}}\\
&\quad\qquad=-\epsil\INT{ \abs{\nabla\ro}^2\frac{\partial}{\partial\ro}\Bigl( \frac{e(\ro,\teta)}{\teta} + \frac{p(\ro,\teta)}{\ro\teta} -s(\ro,\teta)\Bigr)}\\
&\qquad\quad\qquad-\epsil\INT{ {\nabla\ro}\cdot\nabla\teta\frac{\partial}{\partial\teta}\Bigl( \frac{e(\ro,\teta)}{\teta} + \frac{p(\ro,\teta)}{\ro\teta} -s(\ro,\teta)\Bigr)}\\
&\quad\qquad=-\epsil\INT {\abs{\nabla\ro}^2  \frac{1}{\ro\teta}\frac{\partial p(\ro,\teta)}{\partial \ro}  }\\
&\qquad\quad\qquad+\epsil\INT{ {\nabla\ro}\cdot\nabla\teta\frac{1}{\teta^2}\Bigl(e(\ro,\teta) +\ro\frac{\partial e(\ro,\teta)}{\partial \ro}  \Bigr)}.\label{esteps1}
\end{split}
\end{equation}
Due to $\frac{\partial p(\ro,\teta)}{\partial \ro} =\gamma \ro^{\gamma-1} + \teta > 0$, we can put the first term to the left-hand side, while the other one can be estimated using Young's inequality as follows:
\begin{multline}
\epsil\INT{ {\nabla\ro}\cdot\nabla\teta\frac{1}{\teta^2}\Bigl(\frac{\gamma}{\gamma-1} \ro^{\gamma-1} + c_v\teta  \Bigr)}\\
\leq \epsil\frac{\gamma}{\gamma-1}\INT{ \frac{ \ro^{\gamma-1}{\nabla\ro}}{\sqrt{\teta}}\cdot\frac{{\nabla\teta}}{\teta^{3/2}}}+
\epsil c_v\INT{\Bigl(     \frac{{\nabla\ro}\cdot{\nabla\teta} }{\teta}      \Bigr)} \\
\leq \frac{\epsil\delta}{4} \INT{\frac{1}{\teta}\Bigl( 1+  \ro^{\Gamma-2}    \Bigr)\abs{\nabla\ro}^2 }
+ C(\delta)\epsil\INT{\Bigl(\frac{\abs{\nabla\teta}^2}{\teta^3} + \frac{\abs{\nabla\teta}^2}{\teta}  \Bigr)},\label{esteps2}
\end{multline}
provided $\Gamma\geq2\gamma.$ We will choose $\epsil\ll\delta$ so that $C(\delta)\epsil<\frac{\delta}{2}$, so both terms can be pushed to the left-hand side. Other terms coming from \refx{epscast} can be treated similarly as in\cite{NoPo11}, so we get with $C$ independent of the approximative parameters
\begin{equation}\begin{split}
&\INT{\bigl(\kappa(\teta)+\delta \teta^B+\delta \teta^{-1}\bigr) \frac{\abs{\nabla\teta}^{2}}{\teta^2} }+ \tau\INT{ \Phi\carka(\ln\teta)\frac{\left(\det{\teta}\right)^2}{\teta^3} } \\
&\qquad + \lambda\ins{ \inth{\frac{1}{\teta}d(x,\teta)\Theta_0}} +\lambda\INT{  \Bigl(\frac{1}{\teta} \S(\teta,\nabla\u_N):\nabla\u_N+ \frac{\delta}{\teta^2}\Bigr)} \\
&\qquad + \epsil\delta\lambda\INT {  \frac{1}{\teta}(\Gamma\ro^{\Gamma-2}+2)\abs{\nabla\ro}^2 } \leq
C\biggl(1+ \lambda\ins{ \inth{ d(x,\teta)\Theta_0  } }\biggr.\\
\biggl.&\qquad\qquad\qquad\qquad\qquad+\tau\INT{ \frac{\Phi(\ln\teta)}{\teta}  }+\epsil\lambda\INT{\ro s(\ro,\teta)}\biggr). \label{EIfl}
\end{split}\end{equation}
If we sum up the energy inequality \refx{TEBfl}, and the entropy inequality \refx{EIfl} we get
\begin{equation}\begin{split}
&\INT{\bigl(\kappa(\teta)+\delta \teta^B+\delta \teta^{-1}\bigr) \frac{\abs{\nabla\teta}^{2}}{\teta^2} }+ \tau\INT{ \Phi\carka(\ln\teta)\frac{\left(\det{\teta}\right)^2}{\teta^3} } \\
&\qquad+\lambda\ins{\inth{ d(x,\teta)\left(\teta+\frac{\Theta_0}{\teta}\right) }} +\epsil\delta\lambda\INT{\Bigl(\frac{\Gamma}{\Gamma-1}\ro^{\Gamma}+2\ro^{2}\Bigr)}\\
&\qquad  +\INT{  \Bigl(\frac{1}{\teta} \S(\teta,\nabla\u_N):\nabla\u_N+ \frac{\lambda\delta}{\teta^2}\Bigr)} +\tau \INT{ \Phi(\ln\teta) }\\
&\qquad + \epsil\delta\lambda\INT {  \frac{1}{\teta}(\Gamma\ro^{\Gamma-2}+2)\abs{\nabla\ro}^2 } \leq
C\biggl(1+ \lambda\ins{ \inth{ d(x,\teta)\Theta_0  } }\biggr.\\
\biggl.&\qquad+\tau\INT{ \frac{\Phi(\ln\teta)}{\teta}  }+\lambda\INT{ \bigl(\delta\teta^{-1}+\ro\ef\cdot\u_N  \bigr) }\biggr.\\
\biggl.&\qquad\qquad\qquad\qquad\qquad\qquad\qquad\qquad\qquad+\epsil\lambda\INT{\ro s(\ro,\teta)}\biggr) .
\end{split}\end{equation}
The first three terms on the right-hand side can be treated using their counterparts on the left-hand side, while the last one can be estimated as follows
\begin{multline}
\epsil\INT{\left(\ro \log\left( \frac{\teta^{c_v}}{\ro}\right) + \frac{4a}{3}\teta^3\right) }\leq \frac{\epsil\delta}{4}\INT{\ro^2}\\
+\frac{1}{2}\INT{d(x,\teta)\left(\teta+\frac{\Theta_0}{\teta}\right)}+\INT{\kappa(\teta)\frac{\abs{\nabla\teta}^2}{\teta^2}} + C
\end{multline}
yielding
\begin{equation}\begin{split}
&\INT{\bigl(\kappa(\teta)+\delta \teta^B+\delta \teta^{-1}\bigr) \frac{\abs{\nabla\teta}^{2}}{\teta^2} }+ \tau\INT{ \Phi\carka(\ln\teta)\frac{\left(\det{\teta}\right)^2}{\teta^3} } \\
&\qquad+\lambda\ins{\inth{ d(x,\teta)\left(\teta+\frac{\Theta_0}{\teta}\right) }} +\epsil\delta\lambda\INT{\Bigl(\frac{\Gamma}{\Gamma-1}\ro^{\Gamma}+2\ro^{2}\Bigr)}\\
&\qquad  +\INT{  \Bigl(\frac{1}{\teta} \S(\teta,\nabla\u_N):\nabla\u_N+ \frac{\lambda\delta}{\teta^2}\Bigr)} +\tau \INT{ \Phi(\ln\teta) }\\
&\qquad + \epsil\delta\lambda\INT {  \frac{1}{\teta}(\Gamma\ro^{\Gamma-2}+2)\abs{\nabla\ro}^2 } \leq
C\biggl(1+ \lambda\INT{\abs{\ro\ef\cdot\u_N }  }\biggr) \label{FEfl}.
\end{split}\end{equation}
Finally,
\begin{equation}
\begin{split}
& \INT{\abs{\ro\ef\cdot\u_N }  }\leq C\elnorm{\u_N}{2}{6}\elnorm{\ro}{2}{6/5}\\
&\qquad\leq\frac{1}{2} \INT{  \Bigl(\frac{1}{\teta} \S(\teta,\nabla\u_N):\nabla\u_N\Bigr)}+ C\elnorm{\ro}{2}{6/5}^2\\
&\qquad\leq\frac{1}{2} \INT{  \Bigl(\frac{1}{\teta} \S(\teta,\nabla\u_N):\nabla\u_N\Bigr)}+\frac{\epsil\delta}{4}\INT{\ro^2} + C(\epsil,\delta).
\end{split}\end{equation}

We proved above 
\begin{prop}
Any solution to \refx{fixed} $(\u_N,\teta)$ satisfies:
\begin{equation}\begin{split}
&\INT{\bigl(\kappa(\teta)+\delta \teta^B+\delta \teta^{-1}\bigr) \frac{\abs{\nabla\teta}^{2}}{\teta^2} }+ \tau\INT{ \Phi\carka(\ln\teta)\frac{\left(\det{\teta}\right)^2}{\teta^3} } \\
&\qquad+\lambda\ins{\inth{ d(x,\teta)\left(\teta+\frac{\Theta_0}{\teta}\right) }} +\epsil\delta\lambda\INT{\Bigl(\frac{\Gamma}{\Gamma-1}\ro^{\Gamma}+2\ro^{2}\Bigr)}\\
&\qquad  +\INT{  \Bigl(\frac{1}{\teta} \S(\teta,\nabla\u_N):\nabla\u_N+ \frac{\lambda\delta}{\teta^2}\Bigr)} +\tau \INT{ \Phi(\ln\teta) }\\
&\qquad + \epsil\delta\lambda\INT {  \frac{1}{\teta}(\Gamma\ro^{\Gamma-2}+2)\abs{\nabla\ro}^2 } \leq
C(\epsil,\delta),\label{odhadfl}
\end{split}\end{equation}
where the constant on the right-hand side is independent of $\tau,N,\ksi$ and $\ro$ is given by Proposition \ref{CEfx} as $\ro = \ro(\u_N)$.
\end{prop}
Recalling that the solution to our problem fulfills \eqref{fixed} with $\lambda =1$, inequality \eqref{odhadfl} holds for our solution with $\lambda =1$.


\section{Limit passages}

\subsection{Limit $\tau\rightarrow 0^+$}
Now, we will perform the limit passage as $\tau\rightarrow 0^+$. From \refx{odhadfl} we have
\begin{multline}
	\elnorm{\nabla\u_N}{2}{2}+\elnorm{\nabla\teta}{2}{2}+\elnorm{\nabla\teta^{B/2}}{2}{2}
	+\elnorm{\teta}{1}{3B}+ \elnor{\teta}{L^4(\es\times\partial \Omega)} \\+\elnorm{\nabla\teta^{-\frac{1}{2}}}{2}{2}+\elnorm{\teta^{-2}}{2}{2}+\elnorm{\ro}{\Gamma}{3\Gamma}\leq C(\epsil,\delta)\label{ETt}
\end{multline}
and from \refx{CEf}
\begin{align}
 \elnorm{\nabla^2\ro}{q}{q}+\elnorm{\det\ro}{q}{q}\leq C(\epsil,\delta),\label{CEt}
\end{align}
with some $q\in(1,2).$ Moreover, since the velocity belongs to a finite-dimensional space, it is obviously relatively compact. 

We can also test the energy equation by $\Phi(\ln \teta)$ and $\det\Phi(\ln \teta)$, in order to get some additional information depending on $\ksi$ 
\begin{multline}
 \tau\elnorm{\derit\Phi(\ln\teta)}{2}{2}^2 + \tau\elnorm{\Phi(\ln\teta)}{2}{2}^2+\elnorm{\nabla\Phi(\ln\teta)}{2}{2}^2\\
 +\ksi\elnorm{\Phi\carka(\ln\teta)\left( \frac{(\det\teta)^2}{\teta}+\teta{(\det\teta)^2} \right)}{1}{1}\leq C(\epsil,\delta, N).
\end{multline}
Further, we have $\Phi\carka(\ln\teta)\teta^{-1}\geq K>0$, which yields using the structure of $\Phi$
\begin{equation}
 \ksi\elnorm{\det\teta}{2}{2} + \elnorm{\nabla\teta}{2}{2}\leq C(\epsil,\delta, N).
\end{equation}
To summarize, we have the strong and pointwise convergence of the temperature, and we can easily pass to the limit as $\tau\rightarrow 0^+,$ to get at least one solution in the class $\ro\in \soblps \cap L^{p}\left(\es,\:\sobdp\right)$,  $\u_N\in \LwiN$,  $\ln\teta\in\soblps\cap L^{p}\left(\es,\:\sobdp\right)$  satisfying \refx{CEf}, \refx{MEf}, \refx{ENf}, \refx{ETf}  with $\tau=0$. Moreover, we have \refx{odhadfl} with $\tau=0$ and $\lambda=1$, and the following bound
\begin{multline}
\INT{\bigl(\kappa(\teta)+\delta \teta^B+\delta \teta^{-1}\bigr) \frac{\abs{\nabla\teta}^{2}}{\teta^2} } + \epsil\delta\INT {  \frac{1}{\teta}(\Gamma\ro^{\Gamma-2}+2)\abs{\nabla\ro}^2 }\\
 +\INT{  \Bigl(\frac{1}{\teta} \S(\teta,\nabla\u_N):\nabla\u_N+ \frac{\delta}{\teta^2}\Bigr)}
 + \ins{\inth{ d(x,\teta)\frac{\Theta_0}{\teta} }}\\\leq
C\left(1+\epsil \INT{ \ro s(\ro,\teta)}\right),\label{odhadsN}
\end{multline}
with $C$ independent of $N$, $\ksi$, $\delta$, and $\epsil$.

\def\rnx{\rntou{3\Nx}}
\def\bl{\vektor{b}^l}
\def\unx{\u_{\Nx}}


\subsection{Limit $N\rightarrow \infty$}

Next we want to pass to infinity with the dimension of the Galerkin approximation of the velocity. This arguments are very similar to the corresponding counterpart in \cite{FeTP} for $\gamma = \frac 53$ and $d$ independent of the temperature. We can proceed mutatis mutandis as in the above mentioned paper and we therefore skip the detailed considerations here. Recall only that the procedure is based on separate limit passages for $N_t$ and $N_x$ to infinity (the dimension in time and space, respectively) and on switching from the approximative internal energy balance to the approximative entropy equality (between these two limit passages) which changes immediately after the second limit passage into the (approximative) entropy inequality.

Hence, following \cite{FeTP} we can show
 
\begin{prop} \label{prop5}
There exists at least one approximate solution $(\ro,\u,\teta)$ in the class $\ro\in W^{1,3/2}(\es\times\Omega)\cap L^{3/2}(\es; \sobdp )$, $\u\in L^{2}(\es;\sobo),$ $\teta\in L^{2}(\es;\sob)$ satisfying
\begin{equation}
\begin{split}
\derit{\ro} + \sol(\ro \u) -\epsil \Delta \ro + \epsil \ro= \epsil m  \text{ in }\es\times\Omega, \label{CEe} \\
\frac{\partial \ro}{\partial \en} = 0 \text{ on }\es\times\partial\Omega,
\end{split}
\end{equation}
for all $\bfi\in\cinf_0(\es\times\Omega;\rtri)$
\begin{multline}
\ins{\inte {  \left(   \ksi\derit{\u} \cdot\bfi+\derit{(\ro\u)}\cdot\bfi-(\ro\u\otimes\u):\nabla\bfi   \right)   }} \\
+ \ins{\inte {  \Bigl( \S(\teta,\nabla\u):\nabla\bfi- \left(p(\ro,\teta) +\delta (\ro^\Gamma + \ro^2) \right)\sol\bfi \Bigr)}}\label{MEe} \\
= 	\ins{\inte { \left( -\epsil\nabla\ro \cdot\nabla\u\bfi + \frac{1}{2}\epsil(m-\ro)\u\cdot \bfi +\ro\ef\cdot\bfi  \right) }} ,
\end{multline}
 for all $\psi\in C^{\infty}(\es\times\overline{\Omega})$
\begin{equation}
\begin{split}
	&\ins{\inte { \Big(- \left(  \ksi\ln\teta+  \ro s(\ro,\teta)\right)\derit{\psi}-\ro s(\ro,\teta)\u\cdot\nabla\psi\Big)}}\\
&\quad\qquad+\INT{  \bigl(\sol(\ro\u) +\det{\ro} \bigr)\frac{\ro e + p - \ro\teta s}{\ro\teta}\psi  }\\&\quad\qquad
+\INT{    \Bigl( \bigl(\kappa(\teta) + \delta\teta^{B}+\delta\teta^{-1}\bigr)\frac{\nabla\teta}{\teta}\cdot{\nabla\psi} \Bigr) }\\&\quad\qquad
   +  \ins{\inth {  \frac{d(\teta-\Theta_0)}{\teta} \psi }}=  \left\langle \sigma , \psi\right\rangle +\delta\INT{\frac{1}{\teta^2}\psi} \\&\quad\qquad\qquad\qquad\qquad
+ \delta\epsil\INT {  \frac{1}{\teta}(\Gamma\ro^{\Gamma-2}+2)\abs{\nabla\ro}^2 \psi },\label{ETe}
\end{split}
\end{equation}
with
\begin{equation}\sigma\geq \frac{\S(\teta,\nabla\u):\nabla\u}{\teta} + \bigl(\kappa(\teta)+\delta \teta^B+\delta \teta^{-1}\bigr) \frac{\abs{\nabla\teta}^{2}}{\teta^2} \geq 0,\end{equation}
and for all $\psi\in C^{\infty}(\es)$
\begin{multline}
 -\ins{ \inte{\biggl(  ( \ksi+\ro)\frac{\abs{\u}^2}{2} +\ksi\teta + \ro e(\ro,\teta) + \delta\bigl(\frac{\ro^\Gamma}{\Gamma-1}+\ro^2\bigr)  \biggr) }\derit{\psi}}  \\
+ \ins{ \inth{d(x,\teta)(\teta-\Theta_0) }\:\psi\: }+\epsil\delta \INTpsi{ \frac{\Gamma\ro^\Gamma}{\Gamma-1}+2\ro^2 }\\
=\INTpsi{ \ro\ef\cdot\u + \epsil\delta \frac{\Gamma}{\Gamma-1}m\ro^{\Gamma-1} + 2\epsil\delta m \ro + \frac{\delta}{\teta} }  . \label{TEe}
\end{multline}
Moreover, we have for $\psi \in C^\infty(S^1)$ non-negative
\begin{equation}
\begin{split}
& \ins{ \inte{\bigl(  \ksi\ln \teta + \ro s(\ro,\teta)\bigr)} \derit{\psi} }+\ins{ \inth{\frac{1}{\teta}d(x,\teta)\Theta_0}\: \psi\: }\\
&\qquad + \INTpsib{\biggl(\bigl(\kappa(\teta)+\delta \teta^B+\delta \teta^{-1}\bigr) \frac{\abs{\nabla\teta}^{2}}{\teta^2} \biggr)} \\
&\qquad + \INTpsib{  \Bigl(\frac{1}{\teta} \S(\teta,\nabla\u):\nabla\u + \frac{\delta}{\teta^2}\Bigr)} \\
&\qquad + \frac{\epsil\delta}{2}\INTpsi {  \frac{1}{\teta}(\Gamma\ro^{\Gamma-2}+2)\abs{\nabla\ro}^2 }  \\
&\qquad\qquad\quad\leq C\biggl(1 + \epsil\INTpsib{ \ro s(\ro,\teta) } + \ins{ \inth{ d(x,\teta)\Theta_0\:  }\: \psi\:}\biggr),\label{ETIe}
\end{split}
\end{equation}
and
\begin{multline}
\INT{\bigl(\kappa(\teta)+\delta \teta^B+\delta \teta^{-1}\bigr) \frac{\abs{\nabla\teta}^{2}}{\teta^2} }
+\ins{\inth{ d(x,\teta)\left(\teta+\frac{\Theta_0}{\teta}\right) }}\\
 +\epsil\delta\INT{\Bigl(\frac{\Gamma\ro^{\Gamma}}{\Gamma-1}+2\ro^{2}\Bigr)}+\INT{  \Bigl(\frac{ \S(\teta,\nabla\u):\nabla\u}{\teta}+ \frac{\delta}{\teta^2}\Bigr)}\\
   + \epsil\delta\INT {  \frac{1}{\teta}(\Gamma\ro^{\Gamma-2}+2)\abs{\nabla\ro}^2 } \leq
C\biggl(1+ \INT{\abs{\ro\ef\cdot\u }}\biggr).\label{ETIe2}
\end{multline}
\end{prop}

\subsection{Better regularity of the pressure for $\epsil>0$}
In order to pass to the limit with $\epsil\strong 0+$, we have to establish better estimates of the density. Introducing the modified energy
\begin{multline}
\mathcal{E}=\sup\limits_{t\in\es}{\Edt} = \sup\limits_{t\in\es}\int\limits_\Omega \biggl((\ksi+\ro)\frac{\abs{\u}^2}{2} +\ksi(\teta-\ln\teta) \biggr. \\
\biggl. + H(\ro,\teta) +\delta \Bigl( \frac{\ro^\Gamma}{\Gamma-1}+\ro^2\Bigr) \biggr)\dx
\end{multline}
we now want to show that
\begin{equation}
\mathcal{E} + \INT{ \Big(\ro^{\gamma+1}+ \delta(\ro^{\Gamma+1}+\ro^3)\Big)  } \leq C(\delta).\label{betden}
\end{equation} 
Recalling  \refx{ETIe} and \refx{ETIe2} we obtain by already used mean value argument that
$$ \mathcal{E} \leq C\biggl(1+ \ins{\Edt)}+ \INT{\abs{\ro\ef\cdot\u}} +\epsil\INT{ \ro s(\ro,\teta)} \biggr).\label{MVA}$$
Due to the structure of $\mathcal{E}$, we are able to push the last two terms on the right-hand side of \refx{MVA} to its left-hand side. Thus, it remains to estimate the first integral, especially the terms with powers of density $\ro$. For this purpose, we use a specific test function for the momentum equation \refx{MEe}\footnote{Recall that $\mathcal{B}$ stands for the Bogovskii operator ($\sim\sol^{-1}$)}, namely ($M_0 = m |\Omega|$)
$$\Phib = \mathcal{B}\left[\ro-m\right],$$
yielding
\begin{multline}
\INT{ \bigl(p(\ro,\teta)+\delta (\ro^\Gamma+\ro^2)\bigr)\ro }\\=\INT{ \bigl(p(\ro,\teta)+\delta (\ro^\Gamma+\ro^2)\bigr)M_0 }
+\INT{(\ksi+\ro)\u\cdot\det{\Phib}}\\+\INT{ (\ro\u\otimes\u):\nabla\Phib}+\INT{\S(\teta,\nabla\u):\nabla\Phib}\\
 +\INT{\ro\ef\cdot\Phib}-\epsil\INT{\nabla\ro\cdot\nabla\u\Phib}\\ +\frac{\epsil}{2}\INT{(m-\ro)\u\cdot\Phib}.
\end{multline}
In order to estimate the right-hand side we proceed quite similarly as in the heuristic approach. The details can also be found in \cite{FeTP}. We can prove 
\begin{multline}
\sup\limits_{t\in\es}\int\limits_\Omega \bigg((\ksi+\ro)\frac{\abs{\u}^2}{2} +\ksi(\teta-\ln\teta) + H(\ro,\teta) +\delta \Bigl( \frac{\ro^\Gamma}{\Gamma-1}+\ro^2\Bigr)\bigg) \dx\\
+\INT{\bigl(\kappa(\teta)+\delta \teta^B+\delta \teta^{-1}\bigr) \frac{\abs{\nabla\teta}^{2}}{\teta^2} }
+\ins{\inth{ d(x,\teta)\left(\teta+\frac{\Theta_0}{\teta}\right) }}\\
+\INT{  \biggl(\frac{ \S(\teta,\nabla\u):\nabla\u}{\teta}+ \frac{\delta}{\teta^2}\biggr)}+\INT{ \ro^{\gamma+1}}\leq C(\delta),
\end{multline}
with $C$ in particular independent of $\epsil$.

\def\rod{\ro_\delta}
\def\ud{\u_\delta}
\def\tetad{\teta_\delta}

\subsection{Limit $\epsil\strong 0^+$}
The limit passage for $\epsil\strong 0^+$ uses almost the same arguments as the forthcoming limit for $\delta\strong 0^+$, and except the absence of the strong convergence of the initial densities, and the nonlinear boundary condition for the temperature also the same as in \cite[Section~3.6]{singular}. Therefore, we skip it here, and present only the result of this limit.

We obtain for any $\delta,\ksi>0$ a solution $(\rod,\ud,\tetad)$ satisfying the continuity equation in the renormalized sense
\begin{equation}
\INT{  \left(b(\rod)\derit{\psi} + b(\rod)\ud\cdot\nabla\psi+\bigl( b(\rod)-b\carka(\rod)\rod \bigr) \sol\ud\psi\right) }=0,\label{CEd}
\end{equation}
for any $b\in \cinf[0,\infty), b\carka\in \cinf_c[0,\infty)$, and any $\psi\in \cinf(\es\times\overline{\Omega}).$

The momentum equation \refx{MEe} is satisfied with $\epsil=0,$ id est for all $\bfi\in\cinf_0(\es\times\Omega;\rtri)$ we have
\begin{multline}
\ins{\inte {  \left(   (\ksi+\rod)\ud\cdot\derit\bfi+(\rod\ud\otimes\ud):\nabla\bfi   \right)   }} \\
+ \ins{\inte {  \Bigl( \left(p(\rod,\tetad) +\delta (\rod^\Gamma + \rod^2) \right)\sol\bfi \Bigr)}}\label{MEd} \\
= 	\ins{\inte { \bigl( \S(\tetad,\nabla\ud):\nabla\bfi-\rod\ef\cdot\bfi  \bigr) }} .
\end{multline}

The entropy inequality has the form ($\psi\in C^{\infty}(\es\times\overline{\Omega})$, $\psi\geq 0$)
\begin{equation}
\begin{split}
	&\ins{\inte {  \bigl(  \ksi\ln\tetad+  \rod s(\rod,\tetad)\bigr)\derit{\psi}+\rod s(\rod,\tetad)\ud\cdot\nabla\psi}}\\
	&\quad\quad -\INT{    \Bigl( \bigl(\kappa(\tetad) + \delta\tetad^{B}+\delta\tetad^{-1}\bigr)\frac{\nabla\tetad}{\tetad}\cdot{\nabla\psi} \Bigr) }\\
	&\quad\quad+\INTpsib{	\biggl( \frac{\S(\tetad,\nabla\ud):\nabla\ud}{\tetad} + \bigl(\kappa(\tetad)+\delta \tetad^B+\delta \tetad^{-1}\bigr) \frac{\abs{\nabla\tetad}^{2}}{\tetad^2} \biggr)}\\
	&\quad\quad   \leq  \ins{\inth {  \frac{d(\tetad-\Theta_0)}{\tetad} \psi }} -\delta\INT{\frac{1}{\tetad^2}\psi} ,\label{ETd}
\end{split}
\end{equation}
and the total energy balance $(\psi\in \cinf(\es))$
\begin{multline}
 \ins{ \inte{\biggl( (\ksi+\rod)\frac{\abs{\ud}^2}{2} +\ksi\tetad + \rod e(\rod,\tetad) + \delta\Bigl(\frac{\rod^\Gamma}{\Gamma-1}+\rod^2\Bigr) \biggr) }\derit{\psi}}  \\
= \ins{ \inth{d(x,\tetad)(\tetad-\Theta_0) }\:\psi\: }-\INTpsib{\Bigl( \rod\ef\cdot\ud   + \frac{\delta}{\tetad} \Bigr)}  . \label{TEd}
\end{multline}

To simplify our further considerations, let us introduce a positive Radon measure $\sigma_\delta$ (slightly different from $\sigma$ introduced in Proposition \ref{prop5}) satisfying for all $\psi\in\cinf(\es\times\overline{\Omega})$
\begin{equation}
\begin{split}
&\left\langle \sigma_\delta , \psi\right\rangle =-	\ins{\inte {  \bigl(  \ksi\ln\tetad+  \rod s(\rod,\tetad)\bigr)\derit{\psi}+\rod s(\rod,\tetad)\ud\cdot\nabla\psi}}\\
	&\quad\qquad +\INT{    \Bigl( \bigl(\kappa(\tetad) + \delta\tetad^{B}+\delta\tetad^{-1}\bigr)\frac{\nabla\tetad}{\tetad}\cdot{\nabla\psi} \Bigr) }\\
	&\quad\qquad   - \ins{\inth {  \frac{d(\tetad-\Theta_0)}{\tetad} \psi }},\label{ETd2}
\end{split}
\end{equation}
then formula \refx{ETd} reads ($\psi\geq 0$)
\begin{multline}
 \left\langle \sigma_\delta , \psi\right\rangle \geq \INT{ \frac{\S(\tetad,\nabla\ud):\nabla\ud}{\teta}} 	+\delta\INT{\frac{1}{\tetad^2}\psi\:}\\
+\INT{ \bigl(\kappa(\tetad)+\delta \tetad^B+\delta \tetad^{-1}\bigr) \frac{\abs{\nabla\tetad}^{2}}{\tetad^2} \psi\:}.
\end{multline}
Now, we are able to to set $\ksi=\delta$ and perform the last limit passage.


\subsection{Limit $\delta\strong 0+$}
The limit passage for $\delta\strong 0$ is the crucial step in our considerations. First of all, we need to derive estimates independent of the approximative parameter $\delta$. This will be done in the same manner as in the heuristic approach in Section 2.1; the only additional estimates which we need, are those dependent on $\delta$. Combining \refx{ETd} and \refx{TEd} we get
\begin{multline}
	\elnor{\ud}{L^{2}(\sobo) }^2+ 	\elnorm{\nabla(\tetad^{\fra{3}{2}})}{2}{2}^2 + \elnorm{\nabla(\ln\tetad)}{2}{2}^2 \\+\elnor{\frac{1}{\tetad}}{L^{1}(\es\times\partial\Omega)} +
	\elnor{\tetad}{L^{2}(\es\times\partial\Omega)}^2+\delta\elnorm{\nabla(\tetad^{B/2})}{2}{2}^2\\
+\delta \INT{ \frac{1}{\tetad^2}}		\leq C\bigl(1+ \elnorm{\rod}{2}{6/5}^{6/5}\bigr) .
\end{multline}
Further, we can deduce for the (modified) total energy
\begin{multline*}
\Edt = \int\limits_\Omega \Big(\rod\frac{\abs{\ud}^2}{2} + \rod e(\rod,\tetad)  \\
   +\delta \Bigl(\frac{\abs{\ud}^2}{2} +\frac{1}{2}(\tetad+\abs{\ln\tetad})+\frac{\rod^\Gamma}{\Gamma-1}+\rod^2\Bigr) \Big)\dx
\end{multline*}
that
\begin{equation}
\sup\limits_{t\in\es}{\Edt}\leq C \biggl( 1 + \INT{\rod^\gamma + \delta\left(\rod^2+\rod^\Gamma\right)}\biggr).
\end{equation}
Following closely Section 2.1.2 we finally obtain the estimates of the form
\begin{equation}
\begin{split}
\sup\limits_{t\in\es}{\Edt} &+  \INT{\Big(\rod^{\gamma a} + \delta\left(\rod^{2+\gamma(a-1)}+\rod^{\Gamma+\gamma(a-1)}\right)\Big)}\\
 &+\elnor{\tetad}{L^{13/3}(\es\times\partial\Omega)} +\elnor{\ud}{L^{2}(\sobo) }^2+ 	\elnorm{\nabla(\tetad^{\fra{3}{2}})}{2}{2}^2\\
&+ \elnorm{\nabla(\ln\tetad)}{2}{2}^2+\delta\elnorm{\nabla(\tetad^{\fra{B}{2}})}{2}{2}^2  \leq C.\label{finest}
\end{split}
\end{equation}
with $a>1$ (see \refx{posvala}), provided $d$ satisfies \refx{dep}, and $C$ independent of the approximation parameters.

\def\w{\overline}

From these estimates we obtain the quadruple $(\ro,\u,\teta,\sigma)$ such that
\begin{align*}
\rod\weaks\ro\text{ in }&\elka{\infty}{p} ) \text{ and in } \lebs{p} \text{ for some } \:    p>1 ,\\
\ud\weak \u \text{ in }& L^2(\es;\sob) \:\embed \elka{2}{6},\\
\tetad\weaks\teta \text{ in }& \elka{\infty}{4}\text{, and in } L^2\bigl(\es;\sob\bigr),\\
\sigma_\delta\weaks\sigma \text{ in }&\mathcal{M}(\es\times \overline{\Omega}).
\end{align*}
Additionally, we can deduce that
\begin{align}
& \rod\strong\ro \text{ in } \Cw{5/3},\nonumber \\
& \rod\ud\strong\ro\u \text{ in } \Cw{5/4}, \label{weakcon}\\
& b(\rod)\strong\overline{b(\ro)} \text{ in } \Cw{p}. \nonumber
\end{align}
Here and in what follows we denote the weak limit in $\lebs{1}$ of a sequence $g(\rod,\ud,\tetad)$ by the symbol $\w{g(\ro,\u,\teta)}.$ The main difficulty will be to show for certain nonlinear $g$'s that $g(\ro,\u,\teta) = \w{g(\ro,\u,\teta)}.$ Therefore, we need some tools from the theory of compensated compactness developed by Tartar \cite{Tart79} and Murat \cite{Mura78}, namely the following form of celebrated Div-Curl lemma (see also \cite{singular}).

\subsubsection{Strong convergence of temperature}

\begin{lem}\label{divcurl}
Let $\vektor{U}_\delta\weak \vektor{U}$ in $L^p(\rntou{4},\rntou{4})$,  $\vektor{V}_\delta\weak \vektor{V}$ in $L^q(\rntou{4},\rntou{4}), $ with $$\frac{1}{p}+\frac{1}{q}=\frac{1}{s}<1.$$
Suppose further that $\sol \vektor{V}_\delta$ is precompact in $W^{-1,r}(\rntou{4},\R)$, and that $\curl \vektor{V}_\delta$ is precompact in
 $W^{-1,r}(\rntou{4},\rntou{16})$ for some $r\in(1,\infty).$\footnote{Note that the operators $\sol$ and $\curl$ represent here 
their four-dimensional versions in contrast to the rest of the article where they are used in their usual three dimensional sense.} Then $$\vektor{V}_\delta\cdot \vektor{U}_\delta\weak \vektor{V}\cdot \vektor{U} \text{ in }L^s(\rntou{4}).$$
\end{lem}
Let us apply this proposition to the following 4-dimensional vector fields
\begin{align*}
\vektor{V}_\delta:=&\left[\delta\ln\tetad+\rod s(\rod,\tetad),\: \rod s(\rod,\tetad)\ud - \left(\bigl(\kappa(\tetad) + \frac{\delta}{\teta}+ \delta\teta^B\bigr)\frac{\nabla\tetad}{\tetad}  \right)\right]\\
\vektor{U}_\delta:=&\bigl[T_k(\tetad),\:0,\:0,\:0\bigr],
\end{align*}
where we have introduced a concave smooth cut-off function
$$ T_k(z)=kT\left(\frac{z}{k}\right),\quad T(z)=\begin{cases}z\text{ for }z\in[0,1]\\
                                   2 \text{ for }z\in[3,\infty).
\end{cases} $$
The structural property of $s(\ro,\teta)$ together with the estimates \refx{finest} ensures that $\vektor{V}_\delta$ is uniformly bounded in $\lebs{p}$ for some $p>1.$ 
In addition, we observe that the terms in the entropy inequality \refx{ETd} with $\delta$ vanish as $\delta\strong 0$ in sense of weak convergence in $L^p(S^1\times \Omega)$ ($p>1$).
Further, \refx{ETd}, and the estimates below implies that all assumptions of Lemma \ref{divcurl} are satisfied  for $\vektor{V}_\delta$, $\vektor{U}_\delta$, hence 
$$ \w{T_k(\teta)\ln\left( \fra{\teta^{c_v}}{\ro}\right)} +\frac{4a}{3} \w{T_k(\teta)\teta^3}=\w{T_k(\teta)}\:\w{\ln\left( \fra{\teta^{c_v}}{\ro}\right)} +\frac{4a}{3} \w{T_k(\teta)}\:\w{\teta^3}  $$
Since the logarithm is a monotone function we get (see e.g. \cite[Section~10.11]{singular})
 $$ \w{T_k(\teta)\ln\left( \fra{\teta^{c_v}}{\ro}\right)}\geq\w{T_k(\teta)}\:\w{\ln\left( \fra{\teta^{c_v}}{\ro}\right)}, $$
so
$$  \w{T_k(\teta)\teta^3}\leq \w{T_k(\teta)}\:\w{\teta^3}, $$
and using that also $f(z)=z^3$ is monotone we obtain that
$$ \teta^3=\w{\teta^3} $$
from where we can conclude that
\begin{equation} \teta_\delta\strong\teta\text{ a.e. in }\es\times\Omega.\end{equation}

Concerning the nonlinear boundary term $d(x,\teta)\teta$, we have compactness according to the key estimate \refx{hranic}
\begin{equation}
\elnor{\teta}{L^{13/3}(\partial\Omega)}^{\frac 32}\leq C \elnor{\teta^{\fra{3}{2}}}{\sob}.
\end{equation}
Hence, due to the standard interpolation argument  we can conclude the strong convergence in $L^p(\es\times\partial\Omega)$ for all $p<\frac{13}{3}$.

The last step is to show the pointwise convergence of densities in order to identify the limit in the pressure. 
For this purpose we will use nowadays "classical" arguments exploited by Lions \cite{Lion98} and Feireisl\cite{Feir04} including
 the effective viscous flux identity, commutator lemma for Riesz operators, oscillations defect measure or the limit renormalized continuity equation.
Although we use simply the same arguments as in \cite{FeTP}, we present this part here to make the limit passage in this section as self-contained as possible.

\def\sfit{\sfi(t,x)}
\def\iden{\chi_{\Omega}}

\subsubsection{Effective viscous flux identity}
In order to get the weak compactness identity for effective viscous flux, we subtract two identities. The first one is the limit momentum equation tested by 
$\bfi =  \sfi\nabla\Delta^{-1}[T_k(\rod)\iden]$.\footnote{Here and in what follows $\iden$ denotes the characteristic function of the set $\Omega$.}
 The second one is obtained by testing the momentum equation \refx{MEd} by $\bfi =  \sfi\nabla\Delta^{-1}[\w{T_k(\ro_\delta)}\iden]$, and then taking the limit as $\delta\strong0;$
in both cases $\sfi\in \cinf_c(\es\times\Omega)$ is an arbitrary cutoff function. Denoting by
$$\mathcal{R}[v]=\mathcal{F}^{-1}\left[\frac{\xi_i\xi_j}{\abs{\xi}^2} \mathcal{F}(v)(\xi) \right]$$
 the ``double" Riesz transform, we obtain
\begin{equation}
\begin{split}
  &\lim\limits_{\delta\strong0+}\INT{\sfit\Bigl(p(\rod,\tetad)T_k(\rod)-\S(\tetad,\ud):\mathcal{R}[T_k(\rod)\iden] \Bigr)}\\
&\qquad=\INT{\sfit\left(\w{p(\ro,\teta)}\:\w{T_k(\ro)}-\S(\teta,\u):\mathcal{R}[\w{T_k(\ro)}\iden] \right)}\\
&\quad\qquad+\lim\limits_{\delta\strong0+}\INTb\sfit\Bigl( T_k(\rod)\ud\cdot\mathcal{R}[\rod\ud\iden]\Bigr.\\
&\qqquad\qqquad\qqquad\qqquad\Bigl.-\rod(\ud\otimes\ud): \mathcal{R}[T_k(\rod)\iden]  \Bigr)\dx\dt\\
&\quad\qquad-\INT{\sfit \Bigl(\w{ T_k(\ro)}\u\cdot\mathcal{R}[\ro\u\iden]-\ro(\u\otimes\u): \mathcal{R}[\w{T_k(\ro)}\iden]  \Bigr)}.\label{test}
\end{split}
\end{equation}
Now, we will need two commutators lemma, see e.g. \cite[Section~10.17]{singular}.
\begin{lem}\label{com1}
 Let $\vektor{V}_\delta\weak \vektor{V}$ in $L^p(\rtri,\rtri)$, and $w_\delta\weak w$ in $L^q(\rtri)$, with
$$\frac{1}{p}+\frac{1}{q}=\frac{1}{s}<1.$$
Then $$w_\delta\mathcal{R}[\vektor{V}_\delta]-\mathcal{R}[w_\delta]\vektor{V}_\delta\weak w\mathcal{R}[\vektor{V}]-\mathcal{R}[w]\vektor{V}\text{ in }L^s(\rtri,\rtri).$$
\end{lem}
\begin{lem}\label{com2}
 Let $\vektor{V}\in L^p(\rtri,\rtri)$, and $w \in W^{1,q}(\rtri)$, where $r\in(1,3),\:p\in(1,\infty),$
$$\frac{1}{p}+\frac{1}{q}-\frac{1}{3}<\frac{1}{s}<1.$$
Then $$\elnor{\mathcal{R}[w\vektor{V}]-w\mathcal{R}[\vektor{V}]}{W^{a,s}(\rtri)}\leq C \elnor{w}{W^{1,r}(\rtri)}\elnor{\vektor{V}}{L^{q}(\rtri)},$$
with $\frac{a}{3}=\frac{1}{s}+\frac{1}{3}-\frac{1}{p}-\frac{1}{q}$; $W^{a,s}(\rtri)$ denotes the Sobolev-Slobodetskii space.
\end{lem}

From Lemma \ref{com1}, and convergences \refx{weakcon} we obtain 
\begin{multline}
 \INT{\sfit \ud\cdot \Big( T_k(\rod) \mathcal{R}[\rod\ud\iden]-\rod\mathcal{R}[T_k(\rod)\iden]\ud\Big)}\\
\strong\INT{\sfit \u\cdot \left( \w{T_k(\ro)} \mathcal{R}[\ro\u\iden]-\ro\mathcal{R}[\w{T_k(\ro)}\iden]\u\right)},
\end{multline}
hence combining it with \refx{test}
\begin{multline}
  \INT{\sfit\left(\w{p(\ro,\teta)T_k(\ro)}-\w{p(\ro,\teta)}\:\w{T_k(\ro)}\right) }\\
=\INT {\sfit\left(\w{\S(\teta,\u):\mathcal{R}[T_k(\ro)\iden]} -\S(\teta,\u):\mathcal{R}[\w{T_k(\ro)}\iden] \right)}.\label{comu1}
\end{multline}
Further, denoting
\begin{multline}
  \omega(\tetad,\ud)= T_k(\tetad) \Bigl( \mathcal{R}:\left[ \sfit\mu(\tetad)\bigl( \nabla\ud+\nabla\ud^T \bigr) \right]\Bigr.\\
\Bigl.- \sfit \mu(\tetad)  \mathcal{R}:\left[ \nabla\ud+\nabla\ud^T  \right]  \Bigr),
\end{multline}
we get
\begin{multline}
 \INT {\sfit\w{\S(\teta,\u):\mathcal{R}[T_k(\ro)\iden]}}=\lim\limits_{\delta\strong0+}\INT{\omega(\tetad,\ud)}\\
+\lim\limits_{\delta\strong0+}\INT{\sfit\left(\frac{4}{3}\mu(\tetad)+\eta(\tetad)\right)\sol\ud T_k(\rod)  }.\label{comu2}
\end{multline}
According to Lemma \ref{com2}, the vector fields
\begin{equation*}
  \vektor{V}_{\delta}:= [T_k(\rod),T_k(\rod)\ud],\text{ and } \:\vektor{U}_\delta :=[\omega(\tetad,\ud),\:0,\:0,\:0]
\end{equation*}
satisfy the hypotheses of Lemma \ref{divcurl}, hence we obtain
\begin{equation} 
\w{\omega(\teta,\u)}= \omega(\teta,\u).\label{omega}
\end{equation}
From \refx{comu1},\refx{comu2}, and \refx{omega} we finally get the famous effective viscous flux identity
\begin{multline}
\Bigl(\frac{4}{3}\mu(\teta)+\eta(\teta)\Bigr)\left(\w{T_k(\ro)\sol\u}-\w{T_k(\ro)}\sol\u\right)\\
=\w{(\ro^{\gamma}+\ro\teta)T_k(\ro)}-\w{(\ro^{\gamma}+\ro\teta)}\:\w{T_k(\ro)}.\label{EVFI}
\end{multline}

\subsubsection{Oscillations defect measure and limit renormalized continuity equation}

We introduce so called oscillations defect measure (see e.g. \cite[Section~3.7.5]{singular})
$$\mathbf{osc_q}[\rod\strong\ro](\es\times\Omega):= \sup\limits_{k>0}\limsup\limits_{\delta\strong0+} \INT{\abs{T_k(\rod)-T_k(\ro)}^q}$$
and estimate for any $\sfi\in\cinf_c(\es\times\Omega)$, $\sfi\geq0$
\begin{equation*}
 \begin{split}
  &\limsup\limits_{\delta\strong0+} \INT{\sfit\abs{T_k(\rod)-T_k(\ro)}^{\gamma+1}}  \\
  &\qquad\quad\leq \limsup\limits_{\delta\strong0+} \INT{\sfit\bigl( T_k(\rod)-T_k(\ro)\bigr)\bigl( \rod^\gamma-\ro^\gamma\bigr)  } \\
  &\qquad\quad\leq \INT{ \sfit \left(\w{\ro^\gamma T_k(\ro)}-\w{\ro^\gamma}\:\w{ T_k(\ro)}\right)}\\
  &\qquad\quad\leq \INT{ \sfit \left(\w{(\ro^\gamma+\ro\teta) T_k(\ro)}-\w{(\ro^\gamma+\ro\teta)}\:\w{ T_k(\ro)}\right)},
 \end{split}
\end{equation*}
where we have used that $f(z)=z^{\gamma}$ is convex, and $T_k(z)$ concave. The right-hand side of the resulted inequality can be estimated by means of \refx{EVFI} to get
\begin{equation}
\mathbf{osc_{\boldsymbol\gamma+1}}[\rod\strong\ro](\es\times\Omega)\leq C.\label{osc} 
\end{equation}
Having \refx{osc} in hands it is straightforward to get that the limit renormalized continuity equation holds, if we use the following lemma from \cite{singular}.
\begin{lem}
  Let $\Omega\subset \rtri$ be open, and assume that we have a family of solutions $(\rod,\ud)$  to renormalized continuity equation \refx{CEd} such that for some $r>1$
\begin{align}
 \rod\weak\ro          & \text{ in }\lebs{1}\\
\ud\weak\u             & \text{ in }\lebs{r}\\
\nabla\ud\weak\nabla\u & \text{ in }\lebs{r}.
\end{align}
Suppose further that for $\frac{1}{q}<1-\frac{1}{r}$
$$\mathbf{osc_{q}}[\rod\strong\ro](\es\times\Omega)< +\infty. $$
Then the limit functions $(\ro,\u)$ solve the renormalized continuity equation for all $b\in C^1[0,\infty)\cap W^{1,\infty}(0,\infty).$
\end{lem}

This result can be extended by means of Lebesgue dominated convergence theorem up to $b\in C^1[0,\infty)$ with suitable growths. Particularly,
\begin{equation}
 \lim\limits_{\delta\strong0}\INT{T_k(\rod)\sol\ud}-\INT{T_k(\ro)\sol\u}=0\label{Ren}
\end{equation}
Putting \refx{EVFI} and \refx{Ren} together we get
\begin{equation*}
\lim\limits_{k\strong\infty}\INT{ \frac{\w{(\ro^{\gamma}+\ro\teta)T_k(\ro)}-\w{(\ro^{\gamma}+\ro\teta)}\:\w{T_k(\ro)}}{\frac{4}{3}\mu(\teta)+\eta(\teta)}\:}=0,
\end{equation*}
yielding the desired conclusion 
$$\rod\strong \ro \text{ a.e. in }\es\times\Omega .$$
This completes the proof of the case with radiation on the boundary.


\section*{Appendix -- a-priori bounds in the case without radiation on the boundary}

We will consider only $\gamma \in (\frac{8}{5},2)$; the cases $\gamma \geq 2$ could be dealt exactly in the same spirit as in \cite{FeTP}.

First we again obtain from the conservation of mass that
\begin{align}
	\ro\in\elka{\infty}{1}. \label{conservok}
\end{align}
Further, we set $\psi\equiv 1$ in the entropy balance equation in order to get
\begin{multline*}
  \ins{\inte   {\left( \abs{\nabla\u}^2   + \frac{(1+\teta^3)\abs{\nabla\teta}^2}{\teta^2}\right)}  } + \ins{\inth {  \frac{d\Theta_0}{\teta}  }} \\ \leq \ins{\inth {  d}}\leq C ,
\end{multline*}
where we have used the structural properties of $\S$ and $\kve$ together with Korn's inequality. Thus,
\begin{gather}
	\|\u\|_{L^2(W^{1,2}_0)} \leq C,\label{velocok}\\
	\|\nabla(\teta^{\frac{3}{2}})\|_{L^2(L^2)} +\|\nabla(\log\teta)\|_{L^2(L^2)} \leq C.\label{temprok}
\end{gather}
By integrating the total energy balance over the whole time period, we deduce  that
\begin{equation}
\ins { \inth { d(\teta-\Theta_0)}} = \ins{ \inte  { \ro \ef\cdot\u} } \leq C\Bigl(1+ \elnorm{\ro}{2}{\fra{6}{5}} \Bigr), \label{forcingok}
\end{equation}
hence according to properties of $d$ and $\Theta_0$
\begin{equation}
\elnor{\teta}{L^{1}(\es\times\partial\Omega)}  \leq C\Bigl(1+ \elnorm{\ro}{2}{\fra{6}{5}} \Bigr),
\end{equation}
which together with \refx{temprok} by virtue of the Poincar\'{e} inequality yields
\begin{equation}\elnorm{\teta}{1}{9}\leq C(1+ \elnorm{\ro}{2}{\fra{6}{5}} ).\label{petkaok} \end{equation}
Our further considerations will show the boundedness of the density $\ro$  in $L^{a\gamma}(\es\times\Omega)$, with $a=\frac{5\gamma-3}{3\gamma }$, therefore we interpolate as follows ($\frac{5}{6} = \alpha + \frac{1-\alpha}{a\gamma}$,  for $\alpha = \frac{5a\gamma-6}{6a\gamma-6}$)
\begin{equation} \elnorm{\ro}{2}{\frac{6}{5}}^2\leq C\left(\elnorm{\ro}{\infty}{1}\right)\ins{ \Bigl(\inte{ \ro^{a\gamma}}\Bigr)^\frac{1}{3a\gamma-3} }.
\label{interpolok} \end{equation}
Thus,
\begin{equation}
\elnorm{\teta}{1}{9}\leq C\biggl[1+\Bigl(\ins{ \bigl(\inte{ \ro^{a\gamma}}\bigr)^\frac{1}{3a\gamma-3} }\Bigr)^{1/2}\biggr].\label{tetaok}
\end{equation}

Denoting the total energy by $E(t)=\inte{\bigl(\frac{1}{2} \ro\abs{\u}^2+\ro e(\ro,\teta) \bigr)}$, we get from its balance, with the usage of the structural properties of the internal energy, and combining \refx{forcingok} and \refx{interpolok} that for all $t_1,\:t_2\in\es$
\begin{align*}
	E(t_1)-E(t_2)\leq C\biggl(1+ \ins{ E(t)} \biggr),\\
	\sup\limits_{t\in\es}E(t)\leq C \biggl(1+ \ins{ E(t)} \biggr).
\end{align*}
From \refx{conservok} and \refx{velocok} we can bound the kinetic energy
\begin{align}
	\ins{ \inte{\frac{1}{2} \ro\abs{\u}^2 } } \leq&\:C\elnorm{\ro}{\infty}{\fra{3}{2}}\leq \epsil \sup\limits_{t\in\es}\inte{\ro e(\ro,\teta)} + C_{\epsil}(M_0), \end{align}
so we have
\begin{multline}	\sup\limits_{t\in\es}E(t)\leq \:C \biggl(1+  \ins{ \inte { \ro e(\ro,\teta) }  }  \biggr)\\ \leq C \biggl(1+  \ins{ \inte { 
\left(\ro^{\gamma}+ \ro\teta +\teta^{4}  \right) }  }  \biggr).
\end{multline}
The first term on the right-hand side will be left as it is. The last term will be estimated as follows
\begin{equation}
\elnor{\teta}{L^{4}}^4\leq \elnor{\teta}{L^{4}}\elnor{\teta}{L^{4}}^3\leq \elnor{\teta}{L^{4}}  \sup\limits_{t\in\es}E^{3/4}(t) ,\label{tetickaok}
\end{equation}
which gives us using  estimates \refx{petkaok} and \refx{interpolok}
 \begin{equation}
 \begin{split}\ins{ \inte { \teta^{4}}  }\leq & \elnorm{\teta}{1}{4}  \sup\limits_{t\in\es}E^{3/4}(t)\leq C \Bigl( 1+ \elnorm{\ro}{2}{\fra{6}{5}}\Bigr)\sup\limits_{t\in\es}E^{3/4}(t) \\
 \leq& \:   C\biggl[1+\Bigl(\ins{ \Big(\inte{ \ro^{a\gamma}}\Big)^\frac{1}{3a\gamma-3} }\Bigr)^{1/2}\biggr] \sup\limits_{t\in\es}E^{3/4}(t).\label{elfourok}
 \end{split}\end{equation}
Further, from \refx{tetaok}
\begin{align}
	\ins{ \inte{ \ro\teta } } \leq& \elnorm{\ro}{\infty}{\gamma}\elnorm{\teta}{1}{\gamma/\gamma-1}\label{smichok}\\ 
	\leq&  C \sup\limits_{t\in\es}E(t)^{1/\gamma}   \biggl[1+\Bigl(\ins{ \Big(\inte{ \ro^{a\gamma}}\Big)^\frac{1}{3a\gamma-3} }\Bigr)^{1/2}\biggr],
\end{align}
 so we get by means of Young`s inequality
\begin{align}
	\sup\limits_{t\in\es}E(t)\leq C \biggl(1+  \elnor{\ro}{L^{\gamma}(\es\times\Omega)}^{\gamma}+\elnor{\ro}{L^{a\gamma}(\es\times\Omega)}^{2a\gamma/3(a\gamma-1)} \biggr).\label{superok}
\end{align}
 In order to finish the estimates, we have to ensure that the power of the last term on the right-hand side is less than $a\gamma$, id est \begin{equation}
 \frac{4a\gamma}{6(a\gamma-1)}<a\gamma,\label{foryoungok}
 \end{equation}
 which gives us the restriction $a\gamma>\frac{5}{3}$, and therefore (as $a = \frac{5\gamma-3}{3\gamma}$)
\begin{equation}
\gamma>\dfrac{8}{5}.
\end{equation}
 
It remains to deduce suitable estimates of density; this will be done by testing the momentum equation by\footnote{Note that in fact  $\gamma(a-1)=\frac{2\gamma-3}{3}$.}
$$\Phib = \B\left[  \ro^{\gamma(a-1)}-\left\{ \ro^{\gamma(a-1)}\right\}_{\Omega}  \right],$$
where $a$ is as above. From the properties of Bogovskii operator and \refx{conservok} it follows that $\Phib\in L^{\infty}\left( \es,W^{\frac{1}{\gamma(a-1)}}(\Omega)\right),$ and $\left\{ \ro^{\gamma(a-1)}\right\}_{\Omega} \in L^{\infty}(\es)$, so we obtain

\begin{multline*}
	\ins{\inte {  \left(   p(\ro,\teta) \ro^{\gamma(a-1)}  \right)   }}\\
	\leq	\ins{\inte {  \left( -\ro\u\cdot\derit{\Phib}-(\ro\u\otimes\u):\nabla\Phib+ \S(\teta,\nabla\u):\nabla\Phib - \ro\ef\cdot\Phib  \right) }}\\
	 +C\ins{\inte {  \left((  \ro^{\gamma}+\ro\teta   +\teta^{4}  ) \left\{\ro^{\gamma(a-1)}\right\}_\Omega \right) }}. 
	\end{multline*}
	
The terms on the left-hand side of the inequality have a positive sign and give the desired estimate of $\ro^{a\gamma}$, if the right-hand side will be estimated. We present here only the most difficult and restrictive terms.

Using similar arguments as those between \refx{tetickaok} and \refx{superok}  we obtain
\begin{multline}
\ins{\inte {  \left(\teta^{4}  \left\{\ro^{\gamma(a-1)}\right\}_\Omega \right) }}+	\ins{\inte { \ro\teta   \left\{\ro^{\gamma(a-1)}\right\}_\Omega  }}\\
	\leq C\left[1+\elnor{\ro}{L^{a\gamma}(\es\times\Omega)}^{a\gamma/6(a\gamma-1)}  \right] \sup\limits_{t\in\es}E^{3/4}(t),
\end{multline}
since \refx{foryoungok}, we are able to push this term to the left-hand side via Young's inequality.
Further, 
\begin{multline}
	\ins{\inte { (\ro\u\otimes\u):\nabla\Phib} }\leq \elnorm{\u}{2}{6}^2\elnorm{\ro\nabla\Phib}{\infty}{{3}/{2}}\\ \leq C\elnorm{\ro}{\infty}{\gamma}\elnorm{\nabla\Phib}{\infty}{3\gamma/(2\gamma-3)}\leq C \sup\limits_{t\in\es}E(t)^{\frac{1}{\gamma}+\frac{2\gamma-3}{3\gamma}}.
\end{multline}
Note that exactly this point determines the value of $a$ $\left(\frac{3\gamma}{2\gamma-3}=\frac{1}{(a-1)}\right)$, and that $\frac{1}{\gamma}+\frac{2\gamma-3}{3\gamma}<1$. For the term with $\det{\Phib}$, we will use the renormalized equation of continuity with $b(\ro) = \ro^{\gamma(a-1)}$; we obtain two terms, one could be estimated similarly as above, the other as follows\footnote{Indeed, $1= \frac{6+\gamma}{6\gamma}+ \frac{6a-5}{6}$ and $\frac{3\cdot2/(2a-1)}{3-2/(2a-1)}=\frac{6}{6a-5}$.}
\begin{equation*}
\begin{split}
&\ins{ \inte { \ro\u\cdot \B \left[ \ro^{\gamma(a-1)}\sol\u -\left\{\ro^{\gamma(a-1)}\sol\u \right\}_\Omega \right] } } 
 \\&\quad\quad\quad\:\leq \elnorm{\ro\u}{2}{\frac{6\gamma}{6+\gamma}} 
\elnorm{ \B \left[ \ro^{\gamma(a-1)}\sol\u -\left\{\ro^{\gamma(a-1)}\sol\u \right\}_\Omega \right]  } {2}{\frac{6}{6a-5}}
\\&\quad\quad\quad\:\leq \elnorm{\ro}{\infty}{\gamma} \elnorm{\u}{2}{6} \elnorm{\ro^{\gamma(a-1)}\sol\u } {2}{\frac{2}{2a-1}}
\\&\quad\quad\quad\:\leq C  \sup\limits_{t\in\es}E(t)^{\frac{1}{\gamma}} \elnorm{\ro^{\gamma(a-1)}}{\infty}{\frac{1}{a-1}}\elnorm{\sol\u}{2}{2}
\\&\quad\quad\quad\:\leq C  \sup\limits_{t\in\es}E(t)^{\frac{1}{\gamma}+a-1},
\end{split}
\end{equation*}
where $\frac{1}{\gamma}+a-1 = \frac{2}{3}.$

The remaining terms could be estimated by analogy, yielding
$$  \ins{\inte {  \ro^{\gamma+\frac{2\gamma-3}{3}}  }} \leq C \sup\limits_{t\in\es}E(t)^{\beta}$$
for some $\beta<1$. This estimate could be plugged into \refx{super} in order to get the desired bound
$$\sup\limits_{t\in\es}E(t)<\infty.$$ The rest of the proof could be done in the same manner as in \cite{FeTP}, and we omit it here.

\textbf{Acknowledgments:} The work of the first author was supported by the grant SVV-2013-267316.
  
\bibliographystyle{amsplain}  

\bibliography{mybib} 

\end{document}